\documentclass[11pt]{article}
\usepackage{graphicx} 
\usepackage{bbm}
\usepackage{bm}
\usepackage[margin=1in]{geometry}
\usepackage{color}
\usepackage{amsfonts}
\usepackage{algorithm}
\usepackage{algorithmic}
\usepackage{latexsym, amssymb, amsmath, amscd, amsthm, amsxtra}
\usepackage{mathtools}
\usepackage{enumerate}
\usepackage[all]{xy}
\usepackage{mathrsfs}
\usepackage{fancyhdr}
\usepackage{listings}
\usepackage{hyperref}
\usepackage{enumitem}
\usepackage{multirow}
\usepackage{tikz}

\def\11{\mathbbm{1}}

\def\ER{Erd\H{o}s-R\'enyi\ }

\def\Fop{\operatorname{F}}

\newcommand{\Pb}{\mathbb P}
\newcommand{\Qb}{\mathbb Q}

\newtheorem{thm}{Theorem}[section]

\newtheorem{question}[thm]{Question}
\newtheorem{conjecture}[thm]{Heuristic}
\newtheorem{lemma}[thm]{Lemma}

\newtheorem{defn}[thm]{Definition}

\newtheorem{remark}[thm]{Remark}
\newtheorem{DEF}[thm]{Definition}

\numberwithin{equation}{section}



\hypersetup{colorlinks=true,linkcolor=blue}

\title{Improved Computational Lower Bound of Estimation for Multi-Frequency Group Synchronization}

\usepackage{authblk}
\author[1]{Zhangsong Li\thanks{Email: \textit{ramblerlzs@pku.edu.cn}. Partially supported by the National Key R$\&$D program of China (Project No. 2023YFA1010103) and the NSFC Key Program (Project No. 12231002).}}
\affil[1]{School of Mathematical Sciences, Peking University}

\date{\today}
\begin{document}
\maketitle

\begin{abstract}
    We study the computational phase transition in a multi-frequency group synchronization problem, where pairwise relative measurements of group elements are observed across multiple frequency channels and corrupted by Gaussian noise. Using the framework of \emph{low-degree polynomial algorithms}, we analyze the task of estimating the structured signal in such observations. We show that, assuming the low-degree heuristic, in synchronization models over the circle group $\mathsf{SO}(2)$, a simple spectral method is computationally optimal among all polynomial-time estimators when the number of frequencies satisfies $L=n^{o(1)}$. This significantly extends prior work \cite{KBK24+}, which only applied to a fixed constant number of frequencies. Together with known upper bounds on the statistical threshold \cite{PWBM18a}, our results establish the existence of a \emph{statistical-to-computational gap} in this model when the number of frequencies is sufficiently large.
\end{abstract}

\tableofcontents

\section{Introduction}{\label{sec:intro}}

\subsection{Group Synchronization}{\label{subsec:group-synchron}}

The recovery of a hidden structured object from noisy matrix-valued observations is a canonical problem in statistics and machine learning \cite{LV07, HTFF09}. In many such problems, the underlying object exhibits a rich group structure, often reflecting inherent physical symmetries or geometric properties of the data. This is particularly evident in applications such as cryo-electron microscopy (cryo-EM), image analysis, and computer vision. A central task within this framework is \emph{group synchronization}, where one aims to estimate unknown group elements from their noisy pairwise comparisons. Beyond its broad practical relevance, synchronization presents a compelling interplay of algebraic structure and statistical inference, making it a fertile ground for theoretical study.

Orientation estimation in cryo-electron microscopy provides a canonical example of a synchronization problem \cite{SS11}. In cryo-EM, the three-dimensional structure of a biomolecule is reconstructed from its noisy two-dimensional projection images. This reconstruction requires estimating the unknown viewing orientations $\bm g_u \in \mathsf{SO}(3)$ from noisy measurements of their relative alignments $\bm g_u \bm g_v^{-1}$. Other examples include community detection in graphs (which can be cast as synchronization over $\mathbb Z_2$) \cite{DAM16, Mas14, MNS15, MNS18}, multi-reference alignment in signal processing (which involves synchronization over $\mathbb Z_L$) \cite{BCSZ14}, network clock synchronization \cite{GK06}, and many others \cite{CLS12, PBPB15, BCLS20}.

In a general synchronization problem over a group $\bm G$, one aims to recover the group-valued vector $\bm u=(\bm g_1,\ldots,\bm g_n) \in \bm G^n$ from noisy pairwise information about $\bm g_k \bm g_j^{-1}$ for all pairs of $(k,j)$. A natural way to model this is to postulate that we obtain a function of $\bm g_k \bm g_j^{-1}$ corrupted with additive Gaussian noise,
\begin{equation}{\label{eq-group-synchronization}}
    \bm Y_{k,j} = f(\bm g_k \bm g_j^{-1}) + \bm W_{k,j}
\end{equation}
for i.i.d.\ Gaussian random variables $\bm W_{k,j}$. We focus on the specific case of angular synchronization, where the objective is to determine phases $\varphi_1,\ldots,\varphi_n \in [0,2\pi]$ from their noisy relative observation $\varphi_k-\varphi_j \mod 2\pi$ \cite{Sin11, BBS17}. This problem can be seen as synchronization over $\mathsf{SO}(2)$, or equivalently, over the complex circle group $\mathsf{U}(1)=\{ e^{i\varphi}: \varphi\in[0,2\pi) \}$. 
\begin{DEF}[Angular synchronization]{\label{def-group-synchronization}}
    We denote the isomorphic groups by
    \begin{equation}{\label{eq-def-group-S}}
        \mathbb S:=\mathsf{SO}(2) \cong \mathsf{U}(1) \,.
    \end{equation}
    In addition, define the Gaussian orthogonal or unitary ensembles (GOE or GUE, respectively) to be the laws of the following Hermitian random matrices $\bm W$. We have $\bm W \in \mathbb R^{n*n}$ or $\bm W \in \mathbb C^{n*n}$ respectively, and its entries are independent random variables except for being Hermitian $\bm W_{j,k}=\overline{\bm W}_{k,j}$. The off-diagonal entries are real or complex standard Gaussian random variables\footnote{A complex standard Gaussian has the law of $x+iy$ where $x,y \sim \mathcal N(0,\frac{1}{2})$ are independent.}, and the diagonal entries follow real Gaussian distribution with $\bm W_{j,j} \sim \mathcal N(0,2)$ or $\bm W_{j,j} \sim \mathcal N(0,1)$, respectively. We write $\mathsf{GOE}(n)$ and $\mathsf{GUE}(n)$ for the respective laws. We consider the synchronization problem where each pairwise alignment between elements $k$ and $j$ is expressed as $e^{i(\varphi_k-\varphi_j)}$, and the obtained noisy observation is
    \begin{equation}{\label{eq-def-angular-synchronization}}
        \bm Y_{k,j} = \frac{\lambda}{\sqrt{n}} \bm x_k \overline{\bm x}_j + \bm W_{k,j} \,.
    \end{equation}
    Here $\bm x_k=e^{i\varphi_k}$ and $\overline{\bm x}_j=e^{-i\varphi_j}$ denotes the complex conjugate. The scalar parameter $\lambda$ is the signal-to-noise ratio, and $\bm W \sim \mathsf{GUE}(n)$ is Gaussian noise as above. 
\end{DEF}
In this case, the observation can also be seen as a rank-one perturbation of the Wigner random matrix $\bm W$,
\begin{align}{\label{eq-def-spiked-Wigner}}
    \bm Y = \frac{\lambda}{\sqrt{n}} \bm x \bm x^* + \bm W \mbox{ where } \bm x^*=\overline{\bm x}^{\top} \,.
\end{align}
This setup, often called a spiked Wigner model, has been extensively studied in random matrix theory and statistics. It exhibits a sharp phase transition in the feasibility of detecting or estimating the signal $\bm x$, which is governed by a variant of the Baik–Ben Arous–P\'ech\'e (BBP) transition \cite{BBP05, FP07, BN11}. Above the BBP threshold $\lambda>1$ , detection is possible based on the top eigenvalue of the observation matrix. Moreover, the top eigenvector of $\bm Y$ correlates non-trivially with the true signal $\bm x$. Conversely, when $\lambda<1$, neither the largest eigenvalue nor its eigenvector provides reliable information about the signal in the high-dimensional limit. The spectral method that extracts the signal via the leading eigenvector is commonly referred to as \emph{principal component analysis (PCA)}. PCA, however, ignores any structural prior information about the signal such as sparsity or entry-wise positivity. Consequently, the spectral threshold depends only on the $L_2$ norm of the signal $\bm x$, and for some choices of a prior distribution of $\bm x$, the performance of PCA is sub-optimal compared to the algorithms exploiting this structural information about the signal \cite{ZHT06, dGJL07, JL09, MR14}.

Nevertheless, while PCA can be improved upon for certain sparse priors, it remains statistically optimal for many dense priors, where no statistic can surpass the spectral threshold \cite{DAM16, PWBM18a}. Examples of such settings for synchronization problems cast as spiked matrix models include $\mathbb Z_2$ synchronization, angular synchronization, and many other random matrix spiked models. In addition, from a computational perspective, it is shown in \cite{KWB22, MW25} that under mild assumptions on the prior distribution of $\bm x$, a large class of algorithms namely those based on {\em low-degree polynomials} cannot surpass the spectral threshold. This provides strong evidence that the spectral transition represents a fundamental computational barrier for a broad range of efficient algorithms, and suggests the emergence of a statistical-computational gap when the prior distribution of $\bm x$ is sparse.

\subsection{Multi-frequency group synchronization}{\label{subsec:multi-frequency-group-synchron}}

In modern data analysis, the introduction of multiple frequency channels raises a compelling question that can the increased signal information they provide lower the computational threshold for detection or estimation. In this work, we consider obtaining measurements through several frequency channels, which is motivated by the Fourier decomposition of the non-linear objective of the non-unique games problem \cite{BCLS20} (see \cite[Remark~1]{KBK24+} for further details of such connection). In the angular synchronization case, this translates to obtaining the following observations.
\begin{DEF}[Multi-frequency angular synchronization]{\label{def-multi-frequency-group-synchronization}}
    Consider the synchronization problem where the observation constitutes of $L$ matrices, defined by
    \begin{equation}{\label{eq-def-multi-frequency-group-syncron}}
    \left\{ \begin{aligned}
        & \bm Y_1 = \frac{\lambda}{\sqrt{n}} \bm x\bm x^{*} + \bm W_1 \,, \\
        & \bm Y_2 = \frac{\lambda}{\sqrt{n}} \bm x^{(2)} \big(\bm x^{(2)} \big)^{*} + \bm W_2 \,, \\
        & \ldots \\
        & \bm Y_L = \frac{\lambda}{\sqrt{n}} \bm x^{(L)} \big( \bm x^{(L)} \big)^{*} + \bm W_L \,. \end{aligned} \right.
    \end{equation}
    Here $\bm x^{(k)}$ denotes the entrywise $k$-th power, and $\bm W_1,\ldots,\bm W_L$ are independent noise matrices sampled from $\mathsf{GUE}(n)$. 
\end{DEF}

With the scaling above, PCA succeeds in detecting a signal if any one of the $\bm Y_i$ past the spectral threshold $\lambda>1$. One may expect that by combining information over the $L$ frequencies, it ought to be possible to detect the signal reliably once $\lambda>1/\sqrt{L}$. This intuition comes from the fact that given $L$ independent draws of a single frequency, PCA would indeed detect the signal once $\lambda>1/\sqrt{L}$. It is at least known that the above hope is too good to be true: while our intuition would lead us to believe that it should be possible to detect the signal from two frequencies once $\lambda>1/\sqrt{2} \approx 0.707$, actually it is provably information-theoretically impossible for any $\lambda<0.937$ \cite{PWBM18a}. In general, it was shown in \cite{PWBM18a} that it is impossible to detect the signal once 
\begin{equation}{\label{eq-info-lower-bound}}
    \lambda<\sqrt{ \frac{2(L-1)\log(L-1)}{L(L-2)} } \,.
\end{equation}
Conversely, they also showed that, for sufficiently large $L$, there exists an \emph{inefficient} algorithm for detection that succeeds once
\begin{equation}{\label{eq-info-upper-bound}}
    \lambda>\sqrt{ \frac{4\log L}{L-1} } \gg \frac{1}{\sqrt{L}} \,.
\end{equation}
As $L$ grows, these two bounds provide a tight characterization of the scaling $\Theta(\sqrt{\log L/L})$ of the statistical threshold for this problem. Moreover, once $L\geq 11$, the quantity in \eqref{eq-info-upper-bound} is smaller than $1$, and thus there is a computationally inefficient algorithm that is superior to PCA applied to a single frequency. However, as for angular synchronization, it remained unknown whether this algorithm could be made efficient, and more generally what the limitations on computationally efficient algorithms are in this setting. Consequently, this motivates the following question for group synchronization problems in general:
\begin{question}
    Does estimation by an efficient algorithm become possible at a lower signal-to-noise ratio compared to a single-frequency model?
\end{question}
Using non‑rigorous statistical physics methods and numerical simulations, \cite{PWBM18b} initially predicted that the additional frequencies would not lower the computational threshold. A more recent rigorous analysis \cite{YWF25} confirmed the information-theoretic limits by deriving a replica formula for the asymptotic mutual information in the multi‑frequency model. Their analysis of the replica solution also led to conjectured phase transitions for computationally efficient algorithms. However, these predictions rely on the conjectured optimality of approximate message passing (AMP) algorithms, which is known not to capture the true computational barrier in all settings (see e.g., \cite{WEM19}).

In \cite{KBK24+}, computational lower bounds were derived using the low‑degree polynomial framework, showing that for synchronization over $\mathbb S$, the BBP threshold $\lambda>1$ cannot be surpassed by efficient algorithms when the number of frequencies $L$ is any fixed constant. A key limitation of their approach, which relies on reducing the problem to a detection task, is that it inherently restricts the analysis to the regime $L=O(1)$ (see Remark~\ref{rmk-reduce-to-detection} for further discussion). Thus, (as stated again in \cite{BKMR25+}) it remains an intriguing question whether the BBP threshold can be surpassed by efficient algorithms when the number of frequencies $L=\omega(1)$. Before presenting our main result for synchronization over $\mathbb S$, we first formalize the concept of estimation.
\begin{DEF}{\label{def-weak-recovery}}
    We say an estimator $\mathcal X:=\mathcal X(\bm Y_1,\ldots,\bm Y_L) \in \mathbb R^{n*n}$ achieves \emph{weak recovery}, if $\mathcal X$ is Hermitian (i.e., $\mathcal X=\mathcal X^*$)  and 
    \begin{align}{\label{eq-def-weak-recovery}}
        \mathbb E_{\Pb}\Bigg[ \frac{ \langle \mathcal X, \bm x \bm x^{*} \rangle }{ \| \mathcal X \|_{\Fop} \| \bm x \bm x^{*} \|_{\Fop} } \Bigg] \geq c \mbox{ for some constant } c>0 \,.
    \end{align}
\end{DEF}

The main result in this work significantly improves the result in \cite{KBK24+}, suggesting that the BBP threshold represents an inherent computational barrier of all polynomial-time algorithms as long as the number of frequencies $L=n^{o(1)}$.
\begin{thm}{\label{MAIN-THM-informal}}
    Assuming the low-degree conjecture (see Section~\ref{subsec:low-degree-conjecture} for details), when $L=n^{o(1)}$ and $\lambda<1$, any algorithm for weak recovery requires runtime at least $\exp(n^{\Omega(1)})$.
\end{thm}
\begin{remark}{\label{rmk-reduce-to-detection}}
    The work \cite{KBK24+} provide evidence of the hardness for estimation by considering the easier detection problem. Precisely, let $\Pb=\Pb_{n,L,\lambda}$ be the law of $(\bm Y_1,\ldots,\bm Y_L)$ as in \eqref{eq-def-multi-frequency-group-syncron}, and let $\Qb=\Qb_{n,L}$ be the law of $(\bm W_1,\ldots,\bm W_L)$. The authors in \cite{KBK24+} shows that (assuming the low-degree conjecture) it is impossible for an efficient algorithm to distinguish $\Pb$ and $\Qb$ provided that $\lambda<1$ and $L=O(1)$. We would like to remark that this approach fails when $L=\omega(1)$, as in this case simply checking the Frobenius norm $(\|\bm Y_1\|_{\Fop},\ldots,\|\bm Y_L\|_{\Fop})$ could distinguish $\Pb$ and $\Qb$ for any constant $\lambda>0$. Nevertheless, it is still possible that when $L=n^{o(1)}$ and $\lambda<1$, all efficient algorithms cannot distinguish $\Pb$ from $\widetilde{\Qb}$, where under $\widetilde{\Qb}$ we have
    \begin{align}
        \bm Y_1=\frac{\lambda}{\sqrt{n}} \bm x_1 \bm x_1^{*} + \bm W_1,\ \ldots,\ \bm Y_L =\frac{\lambda}{\sqrt{n}} \bm x_L \bm x_L^{*} + \bm W_L \,,
    \end{align}
    and $\bm x_1,\ldots,\bm x_L$ are independent copies of $\bm x$. Providing computational lower bound for the detection problem between $\Pb$ and $\widetilde{\Qb}$ warrants a dedicated investigation, which we leave for future work.
\end{remark}

\begin{remark}
    In \cite{GZ19}, the authors proposed a heuristic method based on a variant of the AMP algorithm for multi-frequency angular synchronization. Their numerical results, obtained for $n=100$ and $L=32$, suggest that their method may surpass the spectral (BBP) threshold. Our analysis implies that for such success to be general, the number of frequencies $L$ must grow at least polynomially with the dimension $n$. A natural question, therefore, is whether there exists a critical scaling $\mathscr L_n$ such that efficient algorithms can outperform the BBP threshold if and only if $L \gg \mathscr L_n$. Determining this precise threshold would likely require substantial new technical insights, so we also leave this as an open problem for future work.
\end{remark}

\subsection{Overview of proof strategy}{\label{subsec:proof-strategy}}

This subsection provides a high-level overview of the proof for our main result, Theorem~\ref{MAIN-THM-informal}. As mentioned in Remark~\ref{rmk-reduce-to-detection}, simply considering the detection task between $\Pb$ and $\Qb$ is not sufficient for our purpose, as we are working in a regime that detection between $\Pb$ and $\Qb$ can be solved efficiently but recovery under $\Pb$ is believed to be impossible by efficient algorithms. Our proof is motivated by a series of recent progress \cite{DD23, GHL26+}, which exploits the fact that recovery is in fact considerably harder than detection \emph{statistically}, in the sense that recovery requires the chi-square divergence $\chi^2(\Pb\|\Qb)$ to be at least $e^{\Theta(n)}$, whereas detection only requires the chi-square divergence to be $\omega(1)$. Intuitively, this stems from the fact that a randomly guessed $\widehat{\bm x}$ has a non-vanishing correlation with the true signal $\bm x$ with probability only $e^{-\Theta(n)}$.

The main contribution of this work is to develop a (weaker) \emph{algorithmic analogue} of the arguments in \cite{DD23, GHL26+} within the low-degree polynomial framework. Informally speaking, we show that even a mild bound of the low-degree advantage (defined precisely in \eqref{eq-def-low-deg-Adv}) is sufficient to rule out all efficient recovery algorithms. We feel that this approach is not only effective for the present problem but also constitutes a general and easily implementable methodology that may be applied to a broader class of problems. The proof is organized into three key steps:
\begin{enumerate}
    \item[(1)] \underline{Reducing to a detection problem with lopsided success probability.} Inspired by \cite{DHSS25}, we will use a reduction argument to show that if weak recovery can be achieved efficiently, then we can find an efficient test between $\Pb$ and $\Qb$ with non-vanishing type-I accuracy and exponentially decaying type-II error (see Lemma~\ref{lem-reduce-to-one-sided-test} for details). 
    \item[(2)] \underline{Providing a mild bound on low-degree advantage via interpolation.} Using a delicate Lindeberg's interpolation argument, we will show that the low-degree advantage $\mathsf{Adv}_{\leq D}(\Pb\|\Qb)$ is \emph{not too large} (see Lemma~\ref{lem-bound-low-deg-Adv} for details).
    \item[(3)] \underline{Arguing by contradiction via algorithmic contiguity}. We will carry out a refined analysis of the algorithmic contiguity framework in \cite{Li25}, which shows that (assuming low-degree conjecture) a mild control on the low-degree advantage suffices to rule out all tests with lopsided success probability (see Theorem~\ref{thm-alg-contiguity} for details), thus contradicting with Step~(1).
\end{enumerate}

\subsection{Notation}{\label{subsec:notations}}
 
We record in this subsection some notation conventions. For two probability measures $\mathbb P$ and $\mathbb Q$, we denote the total variation distance between them by $\mathsf{TV}(\mathbb P,\mathbb Q)$. The chi-squared divergence from $\Pb$ to $\Qb$ is defined as $\chi^2(\Pb \| \Qb)= \mathbb E_{\mathbf{X}\sim\Qb}[ (\frac{\mathrm{d}\Pb}{\mathrm{d}\Qb}(\mathbf X))^2 ]$. For a matrix or a vector $M$, we will use $M^{\top}$ to denote its transpose, and we denote $M^*=\overline{M}^{\top}$. For a $k*k$ matrix $M$, let $\operatorname{det}(M)$ and $\operatorname{tr}(M)$ be the determinant and trace of $M$, respectively. Denote $M \succ 0$ if $M$ is positive definite and $M \succeq 0$ if $M$ is semi-positive definite. In addition, if $M=M^*$ is Hermitian we let $\varsigma_1(M) \geq \varsigma_2(M) \geq \ldots \geq \varsigma_k(M)$ be the eigenvalues of $M$. Denote by $\mathrm{rank}(M)$ the rank of the matrix $M$. Denote $\mathsf{O}(m)$ to be the set of all $m*m$ orthogonal matrices, and denote $\mathsf{SO}(m)=\{ M\in \mathsf{O}(m):\mathrm{det}(M)=1 \}$. For two $k*k$ complex matrices $M_1$ and $M_2$, we define their inner product to be
\begin{align*}
    \big\langle M_1,M_2 \big\rangle:=\sum_{i,j=1}^k M_1(i,j) \overline{M}_2(i,j) \,.
\end{align*}
In particular, if $M_1,M_2$ are Hermitian matrices then $\langle M_1,M_2 \rangle \in \mathbb R$. We also define the Frobenius norm and operator norm respectively by
\begin{align*}
    \| M \|_{\operatorname{F}} = \mathrm{tr}(MM^{*})^{\frac{1}{2}} =  \langle M,M \rangle^{\frac{1}{2}}, \quad
    \| M \|_{\operatorname{op}} = \varsigma_1(M M^{*})^{\frac{1}{2}} \,.
\end{align*}
We will use $\mathbb{I}_{k}$ to denote the $k*k$ identity matrix (and we drop the subscript if the dimension is clear from the context). Similarly, we denote $\mathbb{O}_{k*l}$ the $k*l$ zero matrix and denote $\mathbb{J}_{k*l}$ the $k*l$ matrix with all entries being 1. We will abbreviate $\mathbb O_k = \mathbb O_{k*k}$ and $\mathbb J_k=\mathbb J_{k*k}$.  

For any two positive sequences $\{a_n\}$ and $\{b_n\}$, we write equivalently $a_n=O(b_n)$, $b_n=\Omega(a_n)$, $a_n\lesssim b_n$ and $b_n\gtrsim a_n$ if there exists a positive absolute constant $c$ such that $a_n/b_n\leq c$ holds for all $n$. We write $a_n=o(b_n)$, $b_n=\omega(a_n)$, $a_n\ll b_n$, and $b_n\gg a_n$ if $a_n/b_n\to 0$ as $n\to\infty$. We write $a_n =\Theta(b_n)$ if both $a_n=O(b_n)$ and $a_b=\Omega(b_n)$ hold.

\section{Average-case complexity from low-degree polynomials}{\label{sec:low-degree-framework}}

The low-degree polynomial framework first emerged from the works of \cite{BHK+19, HS17, HKP+17, Hopkins18} and has since been refined and extended in various directions. It has found applications across a broad spectrum of problems, including detection and estimation tasks such as planted clique, planted dense subgraph, community detection, sparse PCA, and graph alignment (see \cite{HS17, HKP+17, Hopkins18, KWB22, SW22, BEH+22, DMW25, DKW+22, MW25, KMW24, DDL25, CDGL24+}), optimization problems like finding maximal independent sets in sparse random graphs \cite{GJW24, Wein22}, refutation problems such as certification of RIP and lift-monotone properties of random regular graphs \cite{WBP16, BKW20, KY24} and constraint satisfaction problems such as random $k$-SAT \cite{BH22}. Additionally, it is conjectured in \cite{Hopkins18} that the failure of degree-$D$ polynomials implies the failure of all ``robust'' algorithms with running time $n^{\widetilde{O}(D)}$ (here $\widetilde{O}$ means having at most this order up to a $\operatorname{poly} \log n$ factor). However, to prevent the readers from being overly optimistic for this conjecture, we point out a recent work \cite{BHJK25} that finds a counterexample for this conjecture. We therefore clarify that the low-degree framework is expected to be optimal for a certain, yet imprecisely defined, class of ``high-dimensional'' problems. Despite these important caveats, we still believe that analyzing low-degree polynomials remains highly meaningful for our setting, as it provides a benchmark for robust algorithmic performance. We refer the reader to the survey \cite{Wein25+} for a more detailed discussion of these subtleties.

\subsection{Low-degree advantage and low-degree heuristic}{\label{subsec:low-degree-conjecture}}

In this subsection, we will focus on applying the framework of low-degree polynomials in the context of high-dimensional hypothesis testing problems. To be more precise, consider the hypothesis testing problem between two probability measures $\Pb$ and $\Qb$ based on the sample $\mathsf Y \in \mathbb R^{N}$. We will be especially interested in asymptotic settings where $N=N_n, \Qb=\Qb_n, \Pb=\Pb_n, \mathsf Y=\mathsf Y_n$ scale with $n$ as $n \to \infty$ in some prescribed way. The standard low-degree polynomial framework primarily focus on the following notions on strong and weak detection.
\begin{defn}[Strong/weak detection]{\label{def-strong-weak-detection}}
    We say an algorithm $\mathcal A$ that takes $\mathsf Y$ as input and outputs either $0$ or $1$ achieves
    \begin{itemize}
        \item {\bf strong detection}, if the sum of type-I and type-II errors $\Qb(\mathcal A(\mathsf Y)=1) + \Pb(\mathcal A(\mathsf Y)=0)$ tends to $0$ as $n\to \infty$. 
        \item {\bf weak detection}, if the sum of type-I and type-II errors is uniformly bounded above by $1-\epsilon$ for some fixed $\epsilon>0$.
    \end{itemize}
\end{defn}

Roughly speaking, the low-degree polynomial approach focuses on understanding the capabilities and limitations of algorithms that can be expressed as low-degree polynomial functions of the input variables (in our case, the entries of $\mathsf Y$). To be more precise, let $\mathcal P_D=\mathcal P_{n,D}$ denote the set of polynomials from $\mathbb R^{N_n}$ to $\mathbb R$ with degree no more than $D$. With a slight abuse of notation, we will often refer to ``a polynomial'' to mean a sequence of polynomials $f=f_n \in \mathcal P_{n,D}$; the degree $D=D_n$ of such a polynomial may scale with $n$. As suggested by \cite{Hopkins18}, the key quantity is the low-degree advantage
\begin{equation}{\label{eq-def-low-deg-Adv}}
    \mathsf{Adv}_{\leq D}(\Pb\|\Qb):= \sup_{f \in \mathcal P_{D}} \left\{ \frac{ \mathbb E_{\Pb}[f] }{ \sqrt{ \mathbb E_{\Qb}[f^2] } } \right\} \,.
\end{equation}
The low degree heuristic, proposed in \cite{Hopkins18}, can be summarized as follows.
\begin{conjecture}[Low-degree conjecture]{\label{conj-low-deg}}
    For ``natural'' high-dimensional hypothesis testing problems between $\Pb$ and $\Qb$, the following statements hold.
    \begin{enumerate}
        \item[(1)] If $\mathsf{Adv}_{\leq D}(\Pb\|\Qb)=O(1)$ as $n\to\infty$, then there exists a constant $C$ such that no algorithm with running time $N^{D/(\log N)^C}$ that achieves strong detection between $\Pb$ and $\Qb$. 
        \item[(2)] If $\mathsf{Adv}_{\leq D}(\Pb\|\Qb)=1+o(1)$ as $n\to\infty$, then there exists a constant $C$ such that no algorithm with running time $N^{D/(\log N)^C}$ that achieves weak detection between $\Pb$ and $\Qb$. 
    \end{enumerate}
\end{conjecture}
Motivated by \cite[Hypothesis~2.1.5 and Conjecture~2.2.4]{Hopkins18} as well as the fact that low-degree polynomials capture the best known algorithms for a wide variety of statistical inference tasks, this heuristic appears to hold for distributions of a specific form that frequently arises in high-dimensional statistics. For further discussion on what types of distributions are suitable for this framework, we refer readers to \cite{Hopkins18, Kunisky21, KWB22, ZSWB22, Wein25+}. We also note that while \cite[Hypothesis 2.1.5]{Hopkins18} formally requires the measures $\Pb$ and $\Qb$ to be fully symmetric (i.e., invariant under any permutation of entries), several recent works (e.g., \cite{MW25b, KVWX23, DDL25}) showed that this framework remains applicable under weaker symmetry assumptions, as is the case for the problem studied in this work.

\subsection{Algorithmic contiguity from low-degree heuristic}{\label{subsec:alg-contiguity}}

The framework in Section~\ref{subsec:low-degree-conjecture} provides a useful tool for probing the computational feasibility of both strong and weak detection. However, as discussed in \cite{Li25}, the failure of strong detection alone is insufficient for many purposes, particularly when we aim to construct reductions between statistical models in regimes where detection is computationally possible. To address this, \cite{Li25} introduces a stronger framework that rules out all \emph{one-sided tests} (i.e., tests satisfying $\Pb(\mathcal A=1)\geq \Omega(1)$ and $\Qb(\mathcal A=1)=o(1)$). Under the low-degree heuristic, it was shown that a bounded low-degree advantage not only rules out strong detection but also precludes any efficient one-sided testing algorithm. The core of this subsection is to provide fine-grained estimates of the arguments in \cite{Li25}, thereby elucidating the precise interplay between the low-degree advantage, the testing error, and the algorithm's running time.

\begin{DEF}{\label{def-one-side-test}}
    We say a test $\mathcal A$ that takes $\mathsf Y$ as input and outputs either $0$ or $1$ is a $(\mathrm T;c;\epsilon)$-test between $\Pb$ and $\Qb$, if the following conditions holds:
    \begin{enumerate}
        \item[(1)] The test $\mathcal A$ can be calculated in time $N^{\mathrm T}$;
        \item[(2)] The type-I accuracy satisfies $\Pb(\mathcal A(\mathsf Y)=1) \geq c$;
        \item[(3)] The type-II error satisfies $\Qb(\mathcal A(\mathsf Y)=1) \leq \epsilon$.
    \end{enumerate}
\end{DEF}

\begin{thm}{\label{thm-alg-contiguity}}
    Assuming Heuristic~\ref{conj-low-deg} holds. Consider the hypothesis testing problem between $\Pb=\Pb_n$ and $\Qb=\Qb_n$. If we have $\mathsf{Adv}_{\leq D_n}(\Pb_n \| \Qb_n)^2 \leq \Delta_n$ for some $\Delta_n>1$, then there exists a constant $C>0$ such that there is no $(\mathrm{T}_n;c_n;\epsilon_n)$-test between $\Pb_n$ and $\Qb_n$, such that (recall that $N_n$ is the dimension of $\mathsf Y_n$)
    \begin{equation}{\label{eq-assumption-time-error}}
        \mathrm{T}_n \leq \frac{ D_n-\log(N_n^2 \Delta_n^2)^C }{ \log(N_n^2 \Delta_n^2)^{C} }, \quad c_n = \Omega(1), \quad \epsilon_n \Delta_n =o(1) \,.
    \end{equation}
\end{thm}

\begin{remark}
    Note that if we take $\Delta_n=O(1)$ and $\epsilon_n=o(1)$, Theorem~\ref{thm-alg-contiguity} reduces to \cite[Theorem~11]{Li25}. Thus, Theorem~\ref{thm-alg-contiguity} can be viewed as a fine-grained extension of \cite[Theorem~11]{Li25}.
\end{remark}

The rest part of this section is devoted to the proof of Theorem~\ref{thm-alg-contiguity}. For the sake of brevity, we will only work with some fixed $n$ throughout the analysis and we simply denote $\Pb_n,\Qb_n,N_n, \Delta_n,c_n,\epsilon_n$ as $\Pb,\Qb,N,\Delta,c,\epsilon$, respectively. Our proof roughly follows the proof of \cite[Theorem~11]{Li25} but we need a more careful quantitative bound. Suppose on the contrary that there is an efficient algorithm $\mathcal A$ that takes $\mathsf Y$ as input and outputs either $0$ or $1$ with
\begin{equation}{\label{eq-contradict-assumption}}
    \Pb(\mathcal A(\mathsf Y)=1) \geq c=\Omega(1) \mbox{ and } \Qb(\mathcal A(\mathsf Y)=0) \geq 1-\epsilon \,.
\end{equation}
To this end, we choose $M=M_n$ such that
\begin{equation}{\label{eq-def-M-n}}
    \Delta \ll M \leq N \Delta, \quad \epsilon M = o(1) \,.
\end{equation}
We first briefly explain our proof ideas. Without losing of generality, we may assume that 
\begin{equation}{\label{eq-D-geq-log-MN}}
    D \geq (\log(N^2\Delta^2))^C \geq (\log(MN))^C \,.
\end{equation} 
The crux of our argument is to consider the following \emph{hidden informative sample} problem.

\begin{defn}{\label{def-hidden-sample}}
    Consider the following hypothesis testing problem: we need to determine whether a sample $(\mathsf Y_1,\mathsf Y_2,\ldots, \mathsf Y_M)$ where each $\mathsf Y_i \in \mathbb R^N$ is generated by
    \begin{itemize}
        \item $\overline{\mathcal H}_0$: we let $\mathsf Y_1,\ldots,\mathsf Y_M$ to be independently sampled from $\Qb_n$.
        \item $\overline{\mathcal H}_1$: we first sample $\kappa \in \{ 1,\ldots,M \}$ uniformly at random, and (conditioned on the value of $\kappa$) we let $\mathsf Y_1,\ldots,\mathsf Y_M$ are independent samples with $\mathsf Y_\kappa$ generated from $\Pb$ and $\{ \mathsf Y_j: j \neq \kappa \}$ generated from $\Qb'$.
    \end{itemize}
    In addition, denote $\overline{\Pb}$ and $\overline{\Qb}$ to be the law of $(\mathsf Y_1,\ldots,\mathsf Y_M)$ under $\overline{\mathcal H}_1$ and $\overline{\mathcal H}_0$, respectively. 
\end{defn}

Note that the dimension of $\overline{\Pb},\overline{\Qb}$ equals $MN$. Assuming that \eqref{eq-contradict-assumption} holds, we see that
\begin{align}
    & \overline{\Qb}\Big( \big( \mathcal A(\mathsf Y_1), \ldots, \mathcal A(\mathsf Y_M) \big) = (0,\ldots,0) \Big) \geq 1-M\epsilon \overset{\eqref{eq-def-M-n}}{=} 1-o(1) \,; \label{eq-behavior-Qb-hidden-sample} \\
    & \overline{\Pb}\Big( \big( \mathcal A(\mathsf Y_1), \ldots, \mathcal A(\mathsf Y_M) \big) \neq (0,\ldots,0) \Big) \geq \Omega(1) \,. \label{eq-behavior-Pb-hidden-sample}
\end{align}
Thus, there is an algorithm that achieves weak detection between $\overline{\Pb}$ and $\overline{\Qb}$. In addition, this test can be calculated in time 
\begin{align*}
    M N^{\mathrm T} \overset{\eqref{eq-D-geq-log-MN}}{\leq} M^{D/\log(MN)^C} N^{\mathrm T} \overset{\eqref{eq-assumption-time-error}}{\leq} (MN)^{D/\log(MN)^C} \,.
\end{align*}
Our next result, however, shows the low-degree advantage $\mathsf{Adv}_{\leq D}( \overline{\Pb}\|\overline{\Qb} )$ is bounded by $1+o(1)$.

\begin{lemma}{\label{lem-low-deg-hardness-hidden-sample}}
    If $\mathsf{Adv}_{\leq D}( \Pb\|\Qb )^2=\Delta$, then $\mathsf{Adv}_{\leq D}( \overline{\Pb}\|\overline{\Qb} )=1+o(1)$.
\end{lemma}
\begin{proof}
    Note that
    \begin{align}
        \frac{ \mathrm{d}\overline{\Pb} }{ \mathrm{d}\overline{\Qb} } (\mathsf Y_1,\ldots,\mathsf Y_M) = \frac{1}{M} \sum_{i=1}^{M} \frac{ \mathrm{d}\overline{\Pb}(\cdot \mid \kappa=i) }{ \mathrm{d}\overline{\Qb} } (\mathsf Y_1,\ldots,\mathsf Y_M) = \frac{1}{M} \sum_{i=1}^{M} \frac{ \mathrm{d}\Pb }{ \mathrm{d}\Qb } (\mathsf Y_i) \,.  \label{eq-explicit-form-LR-hidden-sample}
    \end{align}
    In addition, recall \eqref{eq-def-low-deg-Adv}. Define
    \begin{align*}
        \mathcal N_D=\big\{ f \in \mathcal P_D: f = 0 \mbox{ a.s. under } \Qb \big\} \,.
    \end{align*}
    It is straightforward to verify that $\mathcal N_D$ is a linear subspace of $\mathcal P_D$. Also, consider the inner product $\langle f,g \rangle= \mathbb E_{\Qb}[fg]$, denoting $\mathcal N_D^{\perp}$ to be the orthogonal space of $\mathcal N_D$ in $\mathcal P_D$, then $\mathcal N_D^{\perp}$ can be identified as a (finite-dimensional) Hilbert space where this inner product is non-degenerate (note that $\langle f,f \rangle>0$ for all $f \in \mathcal N_D^{\perp} \setminus \{0\}$). In addition, it is straightforward to check that $1 \in \mathcal N_D^{\perp}$ with $\langle 1,1 \rangle=1$. Thus, by applying the Schmidt orthogonalization procedure started with $1$, we can find a standard orthogonal basis $\{ f_{\alpha}: \alpha \in \Lambda \}$ of the space $\mathcal N_D^{\perp}$ and with $0\in\Lambda$ and $f_{0}=1$. Note that in some cases (e.g., when $\Qb$ is a product measure), this standard orthogonal basis have a closed form; however, for general $\Qb$ the explicit form of $\{ f_{\alpha} \}$ is often intractable. We first show that
    \begin{equation}{\label{eq-low-deg-Adv-trandsform}}
        \mathsf{Adv}_{\leq D}\big( \Pb\|\Qb \big) = \left( \sum_{\alpha\in\Lambda} \mathbb E_{\Pb}[ f_{\alpha}(\mathsf Y) ]^2 \right)^{\frac{1}{2}} \,.
    \end{equation}
    Indeed, as $\mathsf{Adv}_{\leq D}( \Pb\|\Qb )^2=\Delta$, we see that for all $f \in \mathcal N_D$ it must holds that $\mathbb E_{\Pb}[f]=0$ (since otherwise we will have $\frac{\mathbb E_{\Pb}[f]}{\mathbb E_{\Qb}[f^2]}=\infty$). Thus, we have
    \begin{align*}
        \mathsf{Adv}_{\leq D}\big( \Pb\|\Qb \big) = \max_{ f \in \mathcal N_D^{\perp} } \left\{ \frac{\mathbb E_{\Pb}[f]}{\mathbb E_{\Qb}[f^2]} \right\} \,.
    \end{align*}
    For any $f \in \mathcal N_D^{\perp}$, it can be uniquely expressed as
    \begin{align*}
        f=\sum_{ \alpha\in\Lambda } C_{\alpha} f_{\alpha} \,,
    \end{align*}
    where $C_{\alpha}$'s are real constants. Applying Cauchy-Schwartz inequality one gets
    \begin{align*}
        \frac{ \mathbb{E}_{\Pb}[f] }{ \sqrt{\mathbb{E}_{\Qb}[f^2]} } = \frac{ \sum_{ \alpha\in\Lambda } C_{\alpha} \mathbb{E}_{\mathbb{P}}[f_{\alpha}(\mathsf Y)] }{ \sqrt{ \sum_{ \alpha\in\Lambda } C_{\alpha}^2} } \leq \left( \sum_{\alpha\in\Lambda} \mathbb E_{\Pb} [f_{\alpha}(\mathsf Y)]^2 \right)^{1/2} \,,
    \end{align*}
    with equality holds if and only if $C_{\alpha} \propto \mathbb{E}_{\Pb}[f_{\alpha}]$. This yields \eqref{eq-low-deg-Adv-trandsform}. Now, note that $\overline{\Qb}=(\Qb)^{\otimes M}$ is a product measure of $\Qb$, there is a natural standard orthogonal basis under $\overline{\Qb}$, given by
    \begin{align*}
        \left\{ \prod_{i=1}^{M} f_{\alpha_i}(\mathsf Y_i): \alpha_i \in \Lambda, \sum_{i=1}^{M} \operatorname{deg}(f_{\alpha_i}) \leq D \right\} \,.
    \end{align*}
    Thus, similarly as in \eqref{eq-low-deg-Adv-trandsform}, we see that
    \begin{align}
        \Big( \mathsf{Adv}_{\leq D}\big( \overline{\Pb}\|\overline{\Qb} \big) \Big)^2 = \sum_{ \substack{ (\alpha_1,\ldots,\alpha_M): \alpha_i \in \Lambda \\ \sum_{i=1}^{M} \operatorname{deg}(f_{\alpha_i}) \leq D } } \mathbb E_{ \overline{\Pb} }\left[ \prod_{i=1}^{M} f_{\alpha_i}(\mathsf Y_i) \right]^2 \,. \label{eq-low-deg-Adv-overline-relax-1}
    \end{align}
    In addition, using \eqref{eq-explicit-form-LR-hidden-sample}, we see from direct calculation that
    \begin{align}
        \mathbb E_{ \overline{\Pb} }\left[ \prod_{i=1}^{M} f_{\alpha_i}(\mathsf Y_i) \right] &= \mathbb E_{ \overline{\Qb} }\left[ \prod_{i=1}^{M} f_{\alpha_i}(\mathsf Y_i) \cdot \frac{\mathrm{d}\overline{\Pb}}{\mathrm{d}\overline{\Qb}} (\mathsf Y_1,\ldots,\mathsf Y_M) \right] \nonumber \\
        &\overset{\eqref{eq-explicit-form-LR-hidden-sample}}{=} \frac{1}{M} \sum_{i=1}^M \mathbb E_{ \overline{\Qb} }\left[ \prod_{i=1}^{M} f_{\alpha_i}(\mathsf Y_i) \cdot \frac{\mathrm{d}\Pb}{\mathrm{d}\Qb} (\mathsf Y_i) \right] \nonumber \\
        &= \begin{cases}
            1 \,, & (\alpha_1,\ldots,\alpha_M) = (0,\ldots,0) \,; \\
            \frac{1}{M} \mathbb E_{\Pb}\big[ f_{\alpha_j}(\mathsf Y_j) \big] \,, & (\alpha_1,\ldots,\alpha_M) = (0,\ldots,0,\alpha_j,0,\ldots,0) \,; \\
            0 \,, & \mbox{otherwise} \,.
        \end{cases} \label{eq-cal-moment-hidden-sample}
    \end{align}
    Plugging \eqref{eq-cal-moment-hidden-sample} into \eqref{eq-low-deg-Adv-overline-relax-1}, we get that 
    \begin{align*}
        \eqref{eq-low-deg-Adv-overline-relax-1} &= 1 + \sum_{i=1}^{M} \sum_{ \alpha_i \in \Lambda\setminus\{ 0 \} } \left( \frac{1}{M} \mathbb E_{\Pb'}\big[ f_{\alpha_i}(\mathsf Y_i) \big] \right)^2 \\
        &\leq 1 + \frac{1}{M} \sum_{\alpha \in \Lambda}  \mathbb E_{\Pb'}\big[ f_{\alpha}(\mathsf Y_j) \big]^2 = 1 + \frac{1}{M} \cdot \Delta = 1+o(1) \,,
    \end{align*}
    where in the second equality we use \eqref{eq-low-deg-Adv-trandsform} and the assumption that $\mathsf{Adv}_{\leq D}( \Pb\|\Qb )^2=\Delta$, and the last equality follows from \eqref{eq-def-M-n}. This completes our proof.
\end{proof}

We can now finish the proof of Theorem~\ref{thm-alg-contiguity}.
\begin{proof}[Proof of Theorem~\ref{thm-alg-contiguity}]
    Suppose on the contrary that there is a $(\mathrm{T};c;\epsilon)$-test $\mathcal A$ satisfying Definition~\ref{def-one-side-test}. Consider the hypothesis testing problem stated in Definition~\ref{def-hidden-sample}. Using \eqref{eq-behavior-Qb-hidden-sample} and \eqref{eq-behavior-Pb-hidden-sample}, we see that there is an efficient algorithm that runs in time $(MN)^{D/\log(MN)^C}$ and achieves weak detection between $\overline{\Pb}$ and $\overline{\Qb}$, which contradicts with Lemma~\ref{lem-low-deg-hardness-hidden-sample} and Item~(2) in Heuristic~\ref{conj-low-deg}.
\end{proof}

\section{Proof of Theorem~\ref{MAIN-THM-informal}}{\label{sec:proof-main-thm}}

In this section we present the proof of Theorem~\ref{MAIN-THM-informal}.  Note that the problem becomes easier as $\lambda$ increases, so without loss of generality we may assume that
\begin{equation}{\label{eq-assumption-lambda}}
    \Theta(1)=\lambda<1-\Omega(1)
\end{equation}
throughout this section. Recall Definition~\ref{def-multi-frequency-group-synchronization}. The first step of our argument is to introduce external randomness by sampling $\bm Z_1,\ldots,\bm Z_L \sim \mathsf{GUE}(n)$ independent of $(\bm Y_1,\ldots,\bm Y_L)$. Denote $\Pb_{\bullet}$ be the joint law of $(\bm Y_1,\ldots,\bm Y_L) \sim \Pb$ and $(\bm Z_1,\ldots,\bm Z_L)$, and denote $\Qb_{\bullet}$ be the joint law of $(\bm Y_1,\ldots,\bm Y_L) \sim \Qb$ and $(\bm Z_1,\ldots,\bm Z_L)$. The first step is to provide the following upper bound of the low-degree advantage $\mathsf{Adv}_{\leq D}(\Pb_{\bullet}\|\Qb_{\bullet})$, as incorporated in the following lemma. 
\begin{lemma}{\label{lem-bound-low-deg-Adv}}
    Suppose that \eqref{eq-assumption-lambda} holds. We have $\mathsf{Adv}_{\leq D}(\Pb_{\bullet}\|\Qb_{\bullet})^2 =\mathsf{Adv}_{\leq D}(\Pb\|\Qb)^2 \leq \exp(O(L))$ for any $D,L=n^{o(1)}$.
\end{lemma}
\begin{remark}
    We note that our bound on $\mathsf{Adv}_{\leq D}(\Pb\|\Qb)$ coincides with the one in \cite{KBK24+} when $L=O(1)$. The analysis in \cite{KBK24+}, which relies on delicate combinatorial arguments, is inherently limited to constant $L$. In contrast, we employ a more robust interpolation method that allows us to control $\mathsf{Adv}_{\leq D}(\Pb\|\Qb)$ for all $L=n^{o(1)}$, thereby extending the earlier result to a much broader scaling regime.
\end{remark}
The next lemma, however, shows that if weak recovery is possible, then we can find an $(\mathrm T;c;\epsilon)$-test between $\Pb_{\bullet}$ and $\Qb_{\bullet}$ with moderate $\mathrm{T}$ and exponentially small $\epsilon$.
\begin{lemma}{\label{lem-reduce-to-one-sided-test}}
    Suppose that for some $\lambda<1-\delta$ where $\delta>0$ is a constant, given  
    \begin{align*}
        (\bm Y_1,\ldots,\bm Y_L) \sim \Pb_{n,L,\lambda}
    \end{align*}
    there exists an estimator $\mathcal X=\mathcal X(\bm Y_1,\ldots,\bm Y_L)$ that can be calculated in time $n^{\mathrm T}$ and achieves weak recovery in the sense of Definition~\ref{def-weak-recovery}. Then there exists $\delta'>0$ and for $\lambda=1-\delta'$, there exists a $(O(\mathrm T);c;\epsilon)$-test between $\Pb_{\bullet}$ and $\Qb_{\bullet}$ with $c=\Omega(1)$ and $\epsilon=e^{-\Theta(n)}$.
\end{lemma}
Based on Lemmas~\ref{lem-bound-low-deg-Adv} and \ref{lem-reduce-to-one-sided-test}, we can now finish the proof of Theorem~\ref{MAIN-THM-informal}.

\begin{proof}[Proof of Theorem~\ref{MAIN-THM-informal}]
    Suppose on the contrary that for some $\lambda<1-\delta$ where $\delta>0$ is a constant, given  
    \begin{align*}
        (\bm Y_1,\ldots,\bm Y_L) \sim \Pb_{n,L,\lambda}
    \end{align*}
    there exists an estimator $\mathcal X=\mathcal X(\bm Y_1,\ldots,\bm Y_L)$ that can be calculated in time $\exp(\mathsf T')$ for some $\mathsf T'=n^{o(1)}$ and achieves weak recovery in the sense of Definition~\ref{def-weak-recovery}. Using Lemma~\ref{lem-reduce-to-one-sided-test}, we see that for some $\lambda=1-\delta'$, there exists an $(O(\mathsf T');\Omega(1);e^{-\Theta(n)})$-test between $\Pb_{\bullet}$ and $\Qb_{\bullet}$. However, using Lemma~\ref{lem-bound-low-deg-Adv} and Theorem~\ref{thm-alg-contiguity}, we see that there is no $(\mathsf T;c;\epsilon)$-test between $\Pb_{\bullet}$ and $\Qb_{\bullet}$, provided that
    \begin{equation}
        \mathsf T \leq \frac{D-(O(L^2)+4\log n)^C}{(O(L^2)+4\log n)^C}, \quad c=\Omega(1), \quad \epsilon=e^{-\Omega(L)} \,.
    \end{equation}
    By choosing $L=n^{o(1)}$ and $D=L^{4C}\cdot \mathsf T'=n^{o(1)}$, this forms a contraction. Thus, we conclude that any estimator $\mathcal X=\mathcal X(\bm Y_1,\ldots,\bm Y_L)$ that can be calculated in time $\exp(n^{o(1)})$ fails to achieve weak recovery in the sense of Definition~\ref{def-weak-recovery}.
\end{proof}
In the rest part of this section, we give the postponed proofs of Lemmas~\ref{lem-bound-low-deg-Adv} and \ref{lem-reduce-to-one-sided-test}.

\subsection{Proof of Lemma~\ref{lem-bound-low-deg-Adv}}{\label{subsec:proof-lem-3.1}}

We first show that $\mathsf{Adv}_{\leq D}(\Pb_{\bullet} \|\Qb_{\bullet})=\mathsf{Adv}_{\leq D}(\Pb\|\Qb)$. Denote $\mu$ be the law of $(\bm Z_1,\ldots,\bm Z_L)$, we then have $\Pb_{\bullet}=\Pb \otimes \mu$ and $\Qb_{\bullet} = \Qb \otimes \mu$ are both product measures. Denote $\{ f_{\alpha}(\bm Y_1,\ldots,\bm Y_L): \alpha \in \mathfrak S_1 \}$ be the standard orthogonal basis of $\mathcal P_{\leq D}(\bm Y_1,\ldots,\bm Y_L)$ under $\Qb$ such that $f_{\emptyset}=1$, and denote $\{ g_{\beta}(\bm Z_1,\ldots,\bm Z_L): \beta \in \mathfrak S_2 \}$ be the standard orthogonal basis of $\mathcal P_{\leq D}(\bm Z_1,\ldots,\bm Z_L)$ under $\mu$ such that $g_{\emptyset}=1$. There is then a natural standard orthogonal basis under $\Qb_{\bullet}$, given by
\begin{align*}
    \left\{ f_{\alpha} g_{\beta}: \operatorname{deg}(f_{\alpha})+\operatorname{deg}(g_{\beta}) \leq D \right\} \,.
\end{align*}
Similarly as in \eqref{eq-low-deg-Adv-trandsform}, we see that
\begin{align}
    \mathsf{Adv}_{\leq D}\big( \Pb_{\bullet} \| \Qb_{\bullet} \big)^2 = \sum_{ \substack{ \alpha\in\mathfrak S_1,\beta\in\mathfrak S_2 \\ \operatorname{deg}(f_{\alpha})+\operatorname{deg}(g_{\beta}) \leq D } } \mathbb E_{\Pb_{\bullet}}\big[ f_{\alpha} g_{\beta} \big]^2 \,. \label{eq-low-deg-Adv-bullet-relax-1}
\end{align}
In addition, from direct calculation that
\begin{align}
    \mathbb E_{\Pb_{\bullet}}\big[ f_{\alpha} g_{\beta} \big]^2 = \begin{cases}
        \mathbb E_{\Pb}\big[ f_{\alpha} \big] \,, & (\alpha,\beta) = (\alpha,\emptyset) \,; \\
        0 \,, & \mbox{otherwise} \,.
    \end{cases} \label{eq-cal-moment-Pb-bullet}
\end{align}
Plugging \eqref{eq-cal-moment-Pb-bullet} into \eqref{eq-low-deg-Adv-bullet-relax-1}, we get that 
\begin{align*}
    \eqref{eq-low-deg-Adv-overline-relax-1} &= \sum_{ \alpha \in \mathfrak S_1 } \left( \mathbb E_{\Pb}\big[ f_{\alpha} \big] \right)^2 =\mathsf{Adv}_{\leq}(\Pb\|\Qb)^2 \,,
\end{align*}
Thus, we have $\mathsf{Adv}_{\leq D}(\Pb_{\bullet} \|\Qb_{\bullet}) =\mathsf{Adv}_{\leq D}(\Pb\|\Qb)$. It remains to show that $\mathsf{Adv}_{\leq D}(\Pb\|\Qb)^2\leq \exp(O(L))$ for any $D,L=n^{o(1)}$. To this end, we first use standard tools for general Gaussian additive models \cite{KWB22} to derive the following bound on low-degree advantage.
\begin{lemma}{\label{lem-low-deg-Adv-relax-1}}
    Define 
    \begin{equation}{\label{eq-exp-leq-D}}
        \exp_{\leq D}(x) = \sum_{k=0}^{D} \frac{ x^k }{ k! } \,.
    \end{equation}
    We then have 
    \begin{align}{\label{eq-bound-Adv-relax-1}}
        \mathsf{Adv}_{\leq D}(\Pb\|\Qb)^2 \leq \mathbb E_{\theta_1, \ldots,\theta_n \sim \operatorname{Unif}[0,2\pi]}\left\{ \prod_{1 \leq \ell \leq L} \exp_{\leq D}\left( \lambda^2 U_{\ell}^2 \right) \prod_{1 \leq \ell \leq L} \exp_{\leq D}\left( \lambda^2 V_{\ell}^2 \right) \right\} \,.
    \end{align}
    Here 
    \begin{align}{\label{eq-def-U-ell-V-ell}}
        U_{\ell} = \frac{1}{\sqrt{n}} \sum_{t=1}^{n} \sin(\ell\theta_t), \quad V_{\ell} = \frac{1}{\sqrt{n}} \sum_{t=1}^{n} \cos(\ell\theta_t) \,.
    \end{align}
\end{lemma}
\begin{proof}
    Using \cite[Lemma~4.3]{KBK24+}, we have that (below we use $\mathbbm 1_n$ to denote the $n$-dimensional all-one vector)
    \begin{equation*}
        \mathsf{Adv}_{\leq D}(\Pb\|\Qb)^2 = \mathbb E_{\bm x}\left\{ \exp_{\leq D}\left( \frac{\lambda^2}{n} \sum_{1\leq\ell\leq L} | \langle \bm x^{(\ell)}, \mathbbm 1_n \rangle |^2 \right) \right\} \,.
    \end{equation*}
    Note that we have $\bm x=(e^{i\theta_1},\ldots,e^{i\theta_n})$ where $\theta_1,\ldots,\theta_n \sim \operatorname{Unif}[0,2\pi]$, so
    \begin{align*}
        \frac{1}{n} | \langle \bm x^{(\ell)}, \mathbbm 1_n \rangle |^2 = U_{\ell}^2 + V_{\ell}^2 \,,
    \end{align*}
    we then have
    \begin{align*}
        \mathsf{Adv}_{\leq D}(\Pb\|\Qb)^2 &= \mathbb E_{\theta_1,\ldots,\theta_n \sim \operatorname{Unif}[0,2\pi]}\left\{ \exp_{\leq D}\left( \lambda^2 \sum_{1\leq\ell\leq L} (U_{\ell}^2+V_{\ell}^2) \right) \right\} \\
        &\leq \mathbb E_{\theta_1,\ldots,\theta_n \sim \operatorname{Unif}[0,2\pi]}\left\{ \prod_{1 \leq \ell \leq L} \exp_{\leq D}\left( \lambda^2 U_{\ell}^2 \right) \prod_{1 \leq \ell \leq L} \exp_{\leq D}\left( \lambda^2 V_{\ell}^2 \right) \right\} \,,
    \end{align*}
    where the inequality follows from $\exp_{\leq D}(x+y) \leq \exp_{\leq D}(x) \exp_{\leq D}(y)$ for $x,y \geq 0$.
\end{proof}

Provided with Lemma~\ref{lem-bound-low-deg-Adv}, it suffices to bound the right hand side of \eqref{eq-bound-Adv-relax-1} when $\lambda<1-\Omega(1)$. Note that
\begin{align}
    \eqref{eq-bound-Adv-relax-1} &= \mathbb E\left\{ \prod_{1 \leq \ell \leq L} \left( \sum_{ 0 \leq k_{\ell} \leq D } \frac{ \lambda^{2k_{\ell}} U_\ell^{2k_{\ell}} }{ k_{\ell}! } \right) \prod_{1 \leq \ell \leq L} \left( \sum_{ 0 \leq m_{\ell} \leq D } \frac{ \lambda^{2m_{\ell}} V_\ell^{2m_{\ell}} }{ m_{\ell}! } \right) \right\} \nonumber \\
    &= \sum_{ \substack{ 0 \leq k_1,\ldots,k_L \leq D \\ 0 \leq m_1,\ldots,m_L \leq D } } \frac{ \lambda^{ 2(k_1+\ldots+k_L+m_1+\ldots+m_L) } }{ k_1!\ldots k_L! m_1!\ldots m_L! } \mathbb E\left\{ U_1^{2k_1} \ldots U_L^{2k_L} V_1^{2m_1} \ldots V_L^{2m_L} \right\} \,.  \label{eq-bound-Adv-relax-2}
\end{align}
The basic intuition behind our approach of bounding \eqref{eq-bound-Adv-relax-2} is that according to \eqref{eq-def-U-ell-V-ell} and central limit theorem, we should expect that 
\begin{align}{\label{eq-Gaussian-approx-intuition}}
    \left( U_1,\ldots,U_L,V_1,\ldots,V_L \right) \mbox{ behaves like } \left( \zeta_1,\ldots,\zeta_L,\eta_1,\ldots,\eta_L \right) \sim \mathcal N(0,\frac{1}{2}\mathbb I_{2L}) \,.
\end{align}
The key of our proof is to show that such Gaussian approximation is indeed valid for all low-order moments via a delicate Lindeberg’s interpolation argument. To this end, we sample 
\begin{align*}
    \left\{ \zeta_\ell(t), \eta_\ell(t): 1 \leq t \leq n, 1 \leq \ell \leq L \right\}
\end{align*}
i.i.d.\ from $\mathcal N(0,\frac{1}{2})$. 
Define
\begin{equation}{\label{eq-def-U-ell(t)-V-ell(t)}}
    \begin{aligned}
        &U_\ell(t)= \frac{1}{\sqrt{n}}\left( \sum_{ 1 \leq j \leq t } \zeta_\ell(j) + \sum_{t+1 \leq j \leq n} \sin(\ell\theta_j) \right) \,; \\ 
        &V_\ell(t)= \frac{1}{\sqrt{n}}\left( \sum_{ 1 \leq j \leq t } \eta_\ell(j) + \sum_{t+1 \leq j \leq n} \cos(\ell\theta_j) \right) \,.
    \end{aligned}
\end{equation}
And define
\begin{align}{\label{eq-def-F-t}}
    F_t = \sum_{ \substack{ 0 \leq k_1,\ldots,k_L \leq D \\ 0 \leq m_1,\ldots,m_L \leq D } } \frac{ \lambda^{ 2(k_1+\ldots+k_L+m_1+\ldots+m_L) } }{ k_1!\ldots k_L! m_1!\ldots m_L! } \mathbb E\left\{ U_1(t)^{2k_1} \ldots U_L(t)^{2k_L} V_1(t)^{2m_1} \ldots V_L(t)^{2m_L} \right\} \,.
\end{align}
In particular, we have (note that $\{ U_{\ell}(n),V_{\ell}(n):1 \leq \ell \leq L \} \overset{d}{=} \{ \zeta_\ell,\eta_\ell:1 \leq \ell \leq L \}$)
\begin{align*}
    F_0 &= \sum_{ \substack{ 0 \leq k_1,\ldots,k_L \leq D \\ 0 \leq m_1,\ldots,m_L \leq D } } \frac{ \lambda^{ 2(k_1+\ldots+k_L+m_1+\ldots+m_L) } }{ k_1!\ldots k_L! m_1!\ldots m_L! } \mathbb E\left\{ U_1^{2k_1} \ldots U_L^{2k_L} V_1^{2m_1} \ldots V_L^{2m_L} \right\} = \eqref{eq-bound-Adv-relax-2} \,; \\
    F_n &= \sum_{ \substack{ 0 \leq k_1,\ldots,k_L \leq D \\ 0 \leq m_1,\ldots,m_L \leq D } } \frac{ \lambda^{ 2(k_1+\ldots+k_L+m_1+\ldots+m_L) } }{ k_1!\ldots k_L! m_1!\ldots m_L! } \mathbb E\left\{ \zeta_1^{2k_1} \ldots \zeta_L^{2k_L} \eta_1^{2m_1} \ldots \eta_L^{2m_L} \right\} \,.
\end{align*}
We now show that $F_t$ and $F_{t+1}$ are ``close'' in a certain sense, as incorporated in the next lemma.
\begin{lemma}{\label{lem-interpolate-F-t-F-t+1}}
    Suppose that $\lambda\leq 1$ and $D,L=n^{o(1)}$, we then have $F_{t+1}=[1+O(n^{-\frac{3}{2}+o(1)})]F_t$ for all $0 \leq t \leq n-1$.
\end{lemma}
\begin{proof}
    For $0 \leq t \leq n-1$, define 
    \begin{equation}{\label{eq-def-U-ell,*(t)-V-ell,*(t)}}
        \begin{aligned}
            &U_{\ell,*}(t)= \frac{1}{\sqrt{n}}\left( \sum_{ 1 \leq j \leq t } \zeta_\ell(j) + \sum_{t+2 \leq j \leq n} \sin(\ell\theta_j) \right) \,; \\ 
            &V_{\ell,*}(t)= \frac{1}{\sqrt{n}}\left( \sum_{ 1 \leq j \leq t } \eta_\ell(j) + \sum_{t+2 \leq j \leq n} \cos(\ell\theta_j) \right) \,.
        \end{aligned}
    \end{equation}
    It is clear that
    \begin{align*}
        & U_{\ell}(t)=U_{\ell,*}(t)+ \frac{1}{\sqrt{n}} \sin(\ell\theta_{t+1}), \quad V_{\ell}(t)=V_{\ell,*}(t)+ \frac{1}{\sqrt{n}}\cos(\ell\theta_{t+1}) \,; \\
        & U_{\ell}(t+1)=U_{\ell,*}(t)+\frac{1}{\sqrt{n}}\zeta_{\ell}(t+1), \quad V_{\ell}(t+1)=V_{\ell,*}(t)+ \frac{1}{\sqrt{n}} \eta_{\ell}(t+1) \,.
    \end{align*}
    Thus, we have that 
    \begin{align}
        F_t &= \sum_{ \substack{ 0 \leq k_1,\ldots,k_L \leq D \\ 0 \leq m_1,\ldots,m_L \leq D } } \mathbb E\left\{ \prod_{1 \leq \ell \leq L} \frac{ \lambda^{2k_\ell} U_\ell(t)^{2k_\ell} }{ k_\ell! } \cdot \frac{ \lambda^{2m_\ell} V_\ell(t)^{2m_\ell} }{ m_\ell! } \right\} \nonumber \\
        &= \sum_{ \substack{ 0 \leq k_1,\ldots,k_L \leq D \\ 0 \leq m_1,\ldots,m_L \leq D } } \mathbb E\left\{ \prod_{1 \leq \ell \leq L} \frac{ \lambda^{2k_\ell} (U_{\ell,*}(t)+ \frac{1}{\sqrt{n}} \sin(\ell\theta_{t+1}))^{2k_\ell} }{ k_\ell! } \cdot \frac{ \lambda^{2m_\ell} (V_{\ell,*}(t)+ \frac{1}{\sqrt{n}} \cos(\ell\theta_{t+1}))^{2m_\ell} }{ m_\ell! } \right\} \nonumber \\
        &= \sum_{ \substack{ 0 \leq \alpha_1,\ldots,\alpha_L \leq 2D \\ 0 \leq \beta_1,\ldots,\beta_L \leq 2D } } \mathbb E\left\{ \prod_{1 \leq \ell \leq L} \left( \frac{\sin(\ell\theta_{t+1})}{\sqrt{n}} \right)^{\alpha_\ell} \left( \frac{\cos(\ell\theta_{t+1})}{\sqrt{n}} \right)^{\beta_\ell} \right\} \Lambda(\alpha_1,\ldots,\alpha_L;\beta_1,\ldots,\beta_L) \,,  \label{eq-F-t-expand}
    \end{align}
    where $\Lambda(\alpha_1,\ldots,\alpha_L;\beta_1,\ldots,\beta_L)$ is defined to be
    \begin{align}{\label{eq-def-Lambda(alpha-ell-beta-ell)}}
        \sum_{ \substack{ \frac{1}{2}\alpha_\ell \leq k_\ell \leq D \\ \frac{1}{2}\beta_\ell \leq m_\ell \leq D } } \mathbb E\left\{ \prod_{1 \leq \ell \leq L} \frac{ \binom{2k_\ell}{\alpha_\ell} \lambda^{2k_\ell} U_{\ell,*}(t)^{2k_\ell-\alpha_\ell} }{ k_\ell! } \prod_{1 \leq \ell \leq L} \frac{ \binom{2m_\ell}{\beta_\ell} \lambda^{2m_\ell} V_{\ell,*}(t)^{2m_\ell-\beta_\ell} }{ m_\ell! } \right\} \,.
    \end{align}
    Here in the derivation of \eqref{eq-F-t-expand} we also use the independence between 
    \begin{align*}
        \left\{ \sin(\ell\theta_{t+1}), \cos(\ell\theta_{t+1}): 1 \leq \ell \leq L \right\} \mbox{ and } \left\{ U_{\ell,*}(t), V_{\ell,*}(t): 1 \leq \ell \leq L \right\} \,.
    \end{align*}
    Similarly, we have
    \begin{align}
        F_{t+1} = \sum_{ \substack{ 0 \leq \alpha_1,\ldots,\alpha_L \leq 2D \\ 0 \leq \beta_1,\ldots,\beta_L \leq 2D } } \mathbb E\left\{ \prod_{1 \leq \ell \leq L} \left( \frac{\zeta_\ell(t+1)}{\sqrt{n}} \right)^{\alpha_\ell} \left( \frac{\eta_\ell(t+1)}{\sqrt{n}} \right)^{\beta_\ell} \right\} \Lambda(\alpha_1,\ldots,\alpha_L;\beta_1,\ldots,\beta_L) \,.  \label{eq-F-t+1-expand}
    \end{align}
    We first argue that
    \begin{align}{\label{eq-bound-growth-Lambda(alpha-ell-beta-ell)}}
        \Lambda(\alpha_1,\ldots,\alpha_L;\beta_1,\ldots,\beta_L) \leq (16D^2)^{\alpha_1+\ldots+\alpha_L+\beta_1+\ldots+\beta_L} \Lambda(0,\ldots,0;0,\ldots,0) \,.
    \end{align}
    To this end, note that (below we use $\binom{b}{a}\leq b^a$) 
    \begin{align*}
        & \binom{2k_\ell}{\alpha_\ell} U_{\ell,*}(t)^{2k_\ell-\alpha_\ell} \leq \frac{1}{2} \left( (2k_\ell)^{2\lfloor\alpha_\ell/2\rfloor} U_{\ell,*}(t)^{2k_\ell-2\lfloor\alpha_\ell/2\rfloor} + (2k_\ell)^{2\lceil\alpha_\ell/2\rceil} U_{\ell,*}(t)^{2k_\ell-2\lceil\alpha_\ell/2\rceil} \right) \,; \\
        & \binom{2m_\ell}{\beta_\ell} V_{\ell,*}(t)^{2m_\ell-\beta_\ell} \leq \frac{1}{2} \left( (2m_\ell)^{2\lfloor\beta_\ell/2\rfloor} V_{\ell,*}(t)^{2m_\ell-2\lfloor\beta_\ell/2\rfloor} + (2m_\ell)^{2\lceil\beta_\ell/2\rceil} V_{\ell,*}(t)^{2m_\ell-2\lceil\beta_\ell/2\rceil} \right) \,.
    \end{align*}
    We have that 
    \begin{align*}
        &\Lambda(\alpha_1,\ldots,\alpha_L;\beta_1,\ldots,\beta_L) \\
        \leq\ & \frac{1}{2^{2L}} \sum_{ \substack{ \alpha_\ell' \in \{ \lfloor\alpha_\ell/2\rfloor, \lceil\alpha_\ell/2\rceil \} \\ \beta_\ell' \in \{ \lfloor\beta_\ell/2\rfloor, \lceil\beta_\ell/2\rceil \} } } \sum_{ \substack{ \frac{1}{2}\alpha_\ell \leq k_\ell \leq D \\ \frac{1}{2}\beta_\ell \leq m_\ell \leq D } } \mathbb E\left\{ \prod_{1 \leq \ell \leq L} \frac{ \lambda^{2k_\ell} (2k_\ell)^{2\alpha_\ell'} U_{\ell,*}(t)^{2k_\ell-2\alpha_\ell'} }{ k_\ell! } \prod_{1 \leq \ell \leq L} \frac{ \lambda^{2m_\ell} (2m_\ell)^{2\beta_\ell'} V_{\ell,*}(t)^{2m_\ell-2\beta_\ell'} }{ m_\ell! } \right\} \\ 
        \leq\ & \frac{1}{2^{2L}} \sum_{ \substack{ \alpha_\ell' \in \{ \lfloor\alpha_\ell/2\rfloor, \lceil\alpha_\ell/2\rceil \} \\ \beta_\ell' \in \{ \lfloor\beta_\ell/2\rfloor, \lceil\beta_\ell/2\rceil \} } } \sum_{ \substack{ \alpha_\ell' \leq k_\ell \leq D \\ \beta_\ell' \leq m_\ell \leq D } } \mathbb E\left\{ \prod_{1 \leq \ell \leq L} \frac{ \lambda^{2k_\ell} (2k_\ell)^{2\alpha_\ell'} U_{\ell,*}(t)^{2k_\ell-2\alpha_\ell'} }{ k_\ell! } \prod_{1 \leq \ell \leq L} \frac{ \lambda^{2m_\ell} (2m_\ell)^{2\beta_\ell'} V_{\ell,*}(t)^{2m_\ell-2\beta_\ell'} }{ m_\ell! } \right\} \\
        :=\ & \frac{1}{2^{2L}} \sum_{ \substack{ \alpha_\ell' \in \{ \lfloor\alpha_\ell/2\rfloor, \lceil\alpha_\ell/2\rceil \} \\ \beta_\ell' \in \{ \lfloor\beta_\ell/2\rfloor, \lceil\beta_\ell/2\rceil \} } } 2^{ 2(\alpha_1'+\ldots+\alpha_L'+\beta_1'+\ldots+\beta_L') } \overline{\Lambda}(2\alpha_1',\ldots,2\alpha_L';2\beta_1',\ldots,2\beta_L') \,,
    \end{align*}
    where in the second inequality we use the fact that $k_\ell \geq \frac{\alpha_\ell}{2}$ implies that $k_\ell \geq \lceil \frac{\alpha_\ell}{2} \rceil \geq \lfloor \frac{\alpha_\ell}{2} \rfloor$ as $k_\ell \in \mathbb N$.
    Thus, it suffices to show that (note that $\lfloor \frac{\alpha_\ell}{2} \rfloor \leq \lceil \frac{\alpha_\ell}{2} \rceil \leq \alpha_\ell$)
    \begin{align}{\label{eq-bound-growth-Lambda(alpha-ell-beta-ell)-even}}
        \overline{\Lambda}(2\alpha_1,\ldots,2\alpha_L; 2\beta_1,\ldots,2\beta_L) \leq (2D)^{2\alpha_1+\ldots+2\alpha_L+2\beta_1+\ldots+2\beta_L} \Lambda(0,\ldots,0;0,\ldots,0) \,.
    \end{align}
    Note that 
    \begin{align*}
        &\overline{\Lambda}(2\alpha_1,\ldots,2\alpha_L; 2\beta_1,\ldots,2\beta_L) \\
        \leq\ & \sum_{ \substack{ \alpha_\ell \leq k_\ell \leq D \\ \beta_\ell \leq m_\ell \leq D } } \mathbb E\left\{ \prod_{1 \leq \ell \leq L} \frac{ \lambda^{2k_\ell} k_\ell^{2\alpha_\ell} U_{\ell,*}(t)^{2k_\ell-2\alpha_\ell} }{ k_\ell! } \prod_{1 \leq \ell \leq L} \frac{ \lambda^{2m_\ell} m_\ell^{2\beta_\ell} V_{\ell,*}(t)^{2m_\ell-2\beta_\ell} }{ m_\ell! } \right\} \\ 
        \leq\ & \sum_{ \substack{ 0 \leq k_\ell \leq D-\alpha_\ell \\ 0 \leq m_\ell \leq D-\beta_\ell } } \mathbb E\left\{ \prod_{1 \leq \ell \leq L} \frac{ \lambda^{2(k_\ell+\alpha_\ell)} (k_\ell+\alpha_\ell)^{2\alpha_\ell} U_{\ell,*}(t)^{2k_\ell} }{ (k_\ell+\alpha_\ell)! } \prod_{1 \leq \ell \leq L} \frac{ \lambda^{2(m_\ell+\beta_\ell)} (m_\ell+\beta_\ell)^{2\beta_\ell} V_{\ell,*}(t)^{2m_\ell} }{ (m_\ell+\beta_\ell)! } \right\} \\
        \leq\ & \sum_{ \substack{ 0 \leq k_\ell \leq D-\alpha_\ell \\ 0 \leq m_\ell \leq D-\beta_\ell } } \mathbb E\left\{ \prod_{1 \leq \ell \leq L} \frac{ \lambda^{2k_\ell} (2D)^{2\alpha_\ell} U_{\ell,*}(t)^{2k_\ell} }{ k_\ell! } \prod_{1 \leq \ell \leq L} \frac{ \lambda^{2m_\ell} (2D)^{2\beta_\ell}V_{\ell,*}(t)^{2m_\ell} }{ m_\ell! } \right\} \\
        \leq\ & (4D^2)^{\alpha_1+\ldots+\alpha_L+\beta_1+\ldots+\beta_L} \sum_{ \substack{ 0 \leq k_\ell \leq D-\alpha_\ell \\ 0 \leq m_\ell \leq D-\beta_\ell } } \mathbb E\left\{ \prod_{1 \leq \ell \leq L} \frac{ \lambda^{2k_\ell} U_{\ell,*}(t)^{2k_\ell} }{ k_\ell! } \prod_{1 \leq \ell \leq L} \frac{ \lambda^{2m_\ell} V_{\ell,*}(t)^{2m_\ell} }{ m_\ell! } \right\} \\
        \leq\ & (4D^2)^{\alpha_1+\ldots+\alpha_L+\beta_1+\ldots+\beta_L} \Lambda(0,\ldots,0;0,\ldots,0) \,,
    \end{align*}
    yielding \eqref{eq-bound-growth-Lambda(alpha-ell-beta-ell)-even} and thus leading to \eqref{eq-bound-growth-Lambda(alpha-ell-beta-ell)}. Now, as $D=n^{o(1)}$, we have
    \begin{align}
        F_t &\overset{\eqref{eq-F-t-expand}}{=} \Lambda(0,\ldots,0;0,\ldots,0) + \sum n^{-\frac{1}{2}(\alpha_1+\ldots+\alpha_L+\beta_1+\ldots+\beta_L)} \Lambda(\alpha_1,\ldots,\alpha_L;\beta_1,\ldots,\beta_L) \nonumber \\
        &\overset{\eqref{eq-bound-growth-Lambda(alpha-ell-beta-ell)}}{=} (1+n^{-\frac{1}{2}+o(1)}) \Lambda(0,\ldots,0;0,\ldots,0) \,. \label{eq-F-t-approx-Lambda-0}
    \end{align}
    In addition, using \eqref{eq-F-t-expand} and \eqref{eq-F-t+1-expand} and noting that
    \begin{align*}
        &\mathbb E\left\{ \prod_{1 \leq \ell \leq L} \left( \frac{\sin(\ell\theta_{t+1})}{\sqrt{n}} \right)^{\alpha_\ell} \left( \frac{\cos(\ell\theta_{t+1})}{\sqrt{n}} \right)^{\beta_\ell} \right\} - \mathbb E\left\{ \prod_{1 \leq \ell \leq L} \left( \frac{\zeta_\ell(t+1)}{\sqrt{n}} \right)^{\alpha_\ell} \left( \frac{\eta_\ell(t+1)}{\sqrt{n}} \right)^{\beta_\ell} \right\} \\
        =\ &
        \begin{cases}
            0, & \alpha_1+\ldots+\alpha_L+\beta_1+\ldots+\beta_L \leq 2 \,; \\
            n^{ -(\frac{1}{2}+o(1))(\alpha_1+\ldots+\alpha_L+ \beta_1+\ldots+\beta_L) }, & \mbox{otherwise} \,.
        \end{cases}
    \end{align*}
    We then have
    \begin{align*}
        F_{t+1}-F_t &= \sum_{ \sum \alpha_\ell + \sum \beta_\ell \geq 3 } n^{ -(\frac{1}{2}+o(1))(\alpha_1+\ldots+\alpha_L+ \beta_1+\ldots+\beta_L) } \cdot \Lambda(\alpha_1,\ldots,\alpha_L;\beta_1,\ldots,\beta_L) \\
        &\overset{\eqref{eq-bound-growth-Lambda(alpha-ell-beta-ell)}}{=} \sum_{ \sum \alpha_\ell + \sum \beta_\ell \geq 3 } n^{ -(\frac{1}{2}+o(1))(\alpha_1+\ldots+\alpha_L+ \beta_1+\ldots+\beta_L) } \cdot \Lambda(0,\ldots,0;0,\ldots,0) \\
        &= n^{-\frac{3}{2}+o(1)} \Lambda(0,\ldots,0;0,\ldots,0) \overset{\eqref{eq-F-t-approx-Lambda-0}}{=} n^{-\frac{3}{2}+o(1)} F_t \,,
    \end{align*}
    where in the third equality we use the fact that $\#\{ (\alpha_\ell,\beta_\ell)_{1 \leq \ell \leq L}: \sum \alpha_\ell + \sum \beta_\ell =k \} \leq (2L)^k$ and $L=n^{o(1)}$.
\end{proof}
We now provide the proof of Lemma~\ref{lem-bound-low-deg-Adv}.
\begin{proof}[Proof of Lemma~\ref{lem-bound-low-deg-Adv}]
    Using Lemma~\ref{lem-interpolate-F-t-F-t+1}, we have 
    \begin{align*}
        \eqref{eq-bound-Adv-relax-2} &= F_0 = [1+O(n^{-\frac{3}{2}+o(1)})]^n F_n = [1+n^{-\frac{1}{2}+o(1)}] F_n \\
        &= \mathbb E\left\{ \prod_{1 \leq \ell \leq L} \exp_{\leq D}\left( \lambda^2 \zeta_{\ell}^2 \right) \prod_{1 \leq \ell \leq L} \exp_{\leq D}\left( \lambda^2 \eta_{\ell}^2 \right) \right\} \\
        &\leq \prod_{1 \leq \ell \leq L} \mathbb E\left\{ \exp\left( \lambda^2 \zeta_{\ell}^2 \right) \right\} \prod_{1 \leq \ell \leq L} \mathbb E\left\{ \exp\left( \lambda^2 \eta_{\ell}^2 \right) \right\} \,,
    \end{align*}
    where in the second inequality we use the independence in $\{ \zeta_\ell,\eta_\ell:1 \leq \ell \leq L \}$ and the fact that $\exp_{\leq D}(x) \leq \exp(x)$ for any $x \geq 0$. Since $\zeta_\ell,\eta_\ell \overset{i.i.d.}{\sim} \mathcal N(0,\frac{1}{2})$ and $\lambda<1-\Omega(1)$, we then have 
    \begin{align*}
        \mathbb E\left\{ \exp\left( \lambda^2 \zeta_{\ell}^2 \right) \right\}, \ \mathbb E\left\{ \exp\left( \lambda^2 \eta_{\ell}^2 \right) \right\} = O(1) \,,
    \end{align*}
    which implies that $\eqref{eq-bound-Adv-relax-2}=\exp(O(L))$ and leads to the desired result.
\end{proof}

\subsection{Proof of Lemma~\ref{lem-reduce-to-one-sided-test}}{\label{subsec:proof-lem-3.2}}

Our proof is inspired by \cite{DHSS25}. Suppose that for some parameters $\lambda<1-\delta$, given 
\begin{align*}
    (\bm Y_1,\ldots,\bm Y_L) \sim \Pb_{n,L,\lambda} \,,
\end{align*}
there exists an estimator $\mathcal X(\bm Y_1,\ldots,\bm Y_L)$ satisfying Definition~\ref{def-weak-recovery}, i.e., 
\begin{align*}
    \mathbb E_{\Pb}\left[ \frac{ \langle \mathcal X, \bm x \bm x^{*} \rangle }{ \| \mathcal X \|_{\Fop} \| \bm x \bm x^{*} \|_{\Fop} } \right] \geq c \mbox{ for some constant } c>0 \,.
\end{align*}
Using standard Markov inequality, we then have
\begin{align*}
    \Pb\left( \frac{ \langle \mathcal X, \bm x \bm x^{*} \rangle }{ \| \mathcal X \|_{\Fop} \| \bm x \bm x^{*} \|_{\Fop} } \geq \frac{c}{2} \right) = 1-\Pb\left(1-\frac{ \langle \mathcal X, \bm x \bm x^{*} \rangle }{ \| \mathcal X \|_{\Fop} \| \bm x \bm x^{*} \|_{\Fop} }\ge 1-\frac{c}{2}\right) \geq 1-\frac{1-c}{1-\frac{c}{2}} \geq \frac{c}{2} \,.
\end{align*}
We now choose a sufficient small constant $\kappa>0$ such that 
\begin{equation}{\label{eq-choice-kappa}}
    \frac{ \lambda }{ 1+\kappa^2 } < 1-\Omega(1) \,.
\end{equation}
Now, suppose $(\bm Y_1,\ldots,\bm Y_L) \sim \Pb_{n,L(1+\kappa^2)\lambda}$ and recall we introduce external randomness $(\bm Z_1,\ldots,\bm Z_L)$. For $1 \leq \ell \leq L$, we set
\begin{align}
    \bm A_{\ell} = \frac{ 1 }{ \sqrt{1+\kappa^2} } \left( \bm Y_{\ell}+ \kappa \bm Z_{\ell} \right), \quad \bm B_{\ell} = \frac{ 1 }{ \sqrt{1+\kappa^{-2}} } \left( \bm Y_{\ell} - \kappa^{-1} \bm Z_{\ell} \right) \,.  \label{eq-split-into-two-matrices}
\end{align}
The reasons for the definition in \eqref{eq-split-into-two-matrices} is as follows (recall Definition~\ref{def-multi-frequency-group-synchronization}):
\begin{itemize}
    \item Under $(\bm Y_1,\ldots,\bm Y_L) \sim \Qb_{n,L}$, we have 
    \begin{align}
        \bm A_{\ell} = \frac{ 1 }{ \sqrt{1+\kappa^2} } \left( \bm W_{\ell}+ \kappa \bm Z_{\ell} \right), \quad \bm B_{\ell} = \frac{ 1 }{ \sqrt{1+\kappa^{-2}} } \left( \bm W_{\ell} - \kappa^{-1} \bm Z_{\ell} \right) \,.  \label{eq-behavior-Y-1,2-Qb}
    \end{align}
    In particular, $(\bm A_1,\ldots,\bm A_{L};\bm B_1,\ldots,\bm B_L)$ are matrices with i.i.d.\ standard complex normal entries and $(\bm B_1,\ldots,\bm B_L)$ is \underline{independent of} $(\bm A_1,\bm A_1,\ldots,\bm A_L)$.
    \item Under $(\bm Y_1,\ldots,\bm Y_L) \sim\Pb_{n,L,(1+\kappa^2)\lambda}$, we have 
    \begin{align}
        \bm A_\ell = \frac{ \lambda }{ \sqrt{n} } (\bm x^{(\ell)}) (\bm x^{(\ell)})^{*} + \overline{\bm A}_{\ell}, \quad \bm B_{\ell} = \frac{ \lambda \sqrt{ (1+\kappa^2) } }{ \sqrt{n(1+\kappa^{-2})} } (\bm x^{(\ell)}) (\bm x^{(\ell)})^{*} + \overline{\bm B}_{\ell} \,,  \label{eq-behavior-Y-1,2-Pb}
    \end{align}
    where 
    \begin{align*}
        \overline{\bm A}_{\ell} = \frac{ 1 }{ \sqrt{1+\kappa^2} } \left( \bm W_{\ell}+ \kappa \bm Z_{\ell} \right), \quad \overline{\bm B}_{\ell} = \frac{ 1 }{ \sqrt{1+\kappa^{-2}} } \left( \bm W_{\ell} - \kappa^{-1} \bm Z_{\ell} \right) \,.
    \end{align*}
    In particular, $(\overline{\bm A}_1,\ldots,\overline{\bm A}_{L}; \overline{\bm B}_1,\ldots,\overline{\bm B}_L)$ are matrices with i.i.d.\ standard normal entries and $(\overline{\bm B}_1,\ldots,\overline{\bm B}_L)$ is \underline{independent of} $(\bm A_1,\ldots,\bm A_L)$. Also we have $(\bm A_1,\ldots,\bm A_L) \sim \Pb_{n,L,\lambda}$.
\end{itemize}

Now, since $(\bm A_1,\ldots,\bm A_L) \sim \Pb_{n,L,\lambda}$ under $\Pb_{\bullet,n,L,(1+\kappa^2)\lambda}$, we can find an estimator 
\begin{align}{\label{eq-generate-estimator}}
    \mathcal X( \bm A_1,\ldots,\bm A_L ) \mbox{ such that } \Pb_{\bullet}\Bigg( \frac{ \langle \mathcal X, \bm x \bm x^{*} \rangle }{ \| \mathcal X \|_{\Fop} \| \bm x \bm x^{*} \|_{\Fop} } \geq \frac{c}{2} \Bigg) \geq \frac{c}{2} \,.
\end{align}
The next lemma shows that $\langle \mathcal X, \bm B_1 \rangle$ is ``large'' under $\Pb_{\bullet}$ and ``small'' under $\Qb_{\bullet}$.
\begin{lemma}{\label{lem-behavior-Pb-Qb}}
    Suppose we choose $\mathcal X( \bm A_1,\ldots,\bm A_L )$ as in \eqref{eq-generate-estimator}. Then 
    \begin{align}
        &\Pb_{\bullet}\left( \big\langle \mathcal X, \bm B_1 \big\rangle \geq \frac{ c\lambda \| \mathcal X \|_{\Fop} }{ 4 } \sqrt{ \frac{ (1+\kappa^2)n }{ (1+\kappa^{-2}) } } \right) \geq \frac{c}{2}-e^{-\Theta(n)} \,;  \label{eq-behavior-Pb} \\
        &\Qb_{\bullet}\left( \big\langle \mathcal X, \bm B_1 \big\rangle \geq \frac{ c\lambda \| \mathcal X \|_{\Fop} }{ 4 } \sqrt{ \frac{ (1+\kappa^2)n }{ (1+\kappa^{-2}) } } \right) \leq e^{-\Theta(n)} \,.  \label{eq-behavior-Qb}
    \end{align}
\end{lemma}
\begin{proof}
    We first prove \eqref{eq-behavior-Qb}. Recall \eqref{eq-behavior-Y-1,2-Qb}, under $\Qb_{\bullet}$ we have $\bm B_1$ is a standard $\mathsf{GUE}(n)$ matrix independent with $(\bm A_1,\ldots,\bm A_L)$ (and thus also independent with $\mathcal X$). Thus, conditioned on $\mathcal X$ we have
    \begin{align*}
        \big\langle \mathcal X, \bm B_1 \big\rangle \sim \mathcal N\left( 0,\| \mathcal X \|_{\Fop}^2 \right) \,.
    \end{align*}
    Thus, from a simple Gaussian tail inequality we see that \eqref{eq-behavior-Qb} holds.
    
    We then prove \eqref{eq-behavior-Pb}. Recall \eqref{eq-behavior-Y-1,2-Pb}, we can decompose $\langle \mathcal X, \bm B_1 \rangle$ into the following terms:
    \begin{align}
        \big\langle \mathcal X, \bm B_1 \big\rangle &= \frac{ \lambda \sqrt{ (1+\kappa^2) } }{ \sqrt{n(1+\kappa^{-2})} } \cdot \big\langle \mathcal X, \bm x \bm x^{*} \big\rangle \label{eq-behavior-Pb-part-1}  \\
        &+ \big\langle \mathcal X, \overline{\bm B}_1  \big\rangle \,.  \label{eq-behavior-Pb-part-2}
    \end{align}
    Using \eqref{eq-generate-estimator}, we see that (note that $\| \bm x \bm x^{*} \|_{\Fop}=\| \bm x \|^2=n$)  
    \begin{align}
        & \Pb_{\bullet}\left( \eqref{eq-behavior-Pb-part-1} \geq \frac{ c\lambda \| \mathcal X \|_{\Fop} }{ 2 } \sqrt{ \frac{ (1+\kappa^2)n }{ (1+\kappa^{-2}) } } \right) = \Pb_{\bullet}\left( \big\langle \mathcal X, \bm x \bm x^{*} \big\rangle \geq \frac{cn\| \mathcal X \|_{\Fop}}{2} \right) \nonumber \\
        \ge \ & \Pb_{\bullet}\left( \big\langle \mathcal X, \bm x \bm x^{*} \big\rangle \geq \frac{c}{2} \| \mathcal X \|_{\Fop} \| \bm x \bm x^{*} \|_{\Fop} \right) \geq \frac{c}{2} \,.  \label{eq-behavior-Pb-part-1-bound}
    \end{align}
    where the last inequality follows from \eqref{eq-generate-estimator}.
    In addition, from the proof of \eqref{eq-behavior-Qb}, we see that
    \begin{align}
        \Pb_{\bullet}\left( |\eqref{eq-behavior-Pb-part-2}| \geq \frac{ c\lambda \| \mathcal X \|_{\Fop} }{ 4 } \sqrt{ \frac{ (1+\kappa^2)n }{ (1+\kappa^{-2}) } } \right) \leq e^{-\Theta(n)} \,.   \label{eq-behavior-Pb-part-2-bound}
    \end{align}
    Combining \eqref{eq-behavior-Pb-part-1-bound} and \eqref{eq-behavior-Pb-part-2-bound}, we have that 
    \begin{align*}
    &\Pb_{\bullet}\left( \big\langle \mathcal X, \bm B_1 \big\rangle \geq \frac{ c\lambda \| \mathcal X \|_{\Fop} }{ 4 } \sqrt{ \frac{ (1+\kappa^2)n }{ (1+\kappa^{-2}) } } \right) \\
    \ge\ & \Pb_{\bullet}\left(\eqref{eq-behavior-Pb-part-1} \geq \frac{ c\lambda \| \mathcal X \|_{\Fop} }{ 2 } \sqrt{ \frac{ (1+\kappa^2)n }{ (1+\kappa^{-2}) } } \right) \\ 
    +\ &\Pb_{\bullet}\left( |\eqref{eq-behavior-Pb-part-2}| < \frac{ c\lambda \| \mathcal X \|_{\Fop} }{ 4 } \sqrt{ \frac{ (1+\kappa^2)n }{ (1+\kappa^{-2}) } } \right)\ge \frac{c}{2}-e^{-\Theta(n)} \,. 
    \end{align*}
    We then see that \eqref{eq-behavior-Pb} holds.
\end{proof}
We point out that Lemma~\ref{lem-behavior-Pb-Qb} immediately implies the existence of an $(O(\mathrm T);\Omega(1);e^{-\Theta(n)})$-test between $\Pb_{\bullet}$ and $\Qb_{\bullet}$.



\bibliographystyle{alpha}

\small
\begin{thebibliography}{10} 

\bibitem[BBP05]{BBP05}
Jinho Baik, G\'erard Ben Arous, and Sandrine P\'ech\'e. 
\newblock Phase transition of the largest eigenvalue for nonnull complex sample covariance matrices. 
\newblock {\em Annals of Probability}, 33(5):1643--1697, 2005.

\bibitem[BBS17]{BBS17}
Afonso S. Bandeira, Nicolas Boumal, and Amit Singer. 
\newblock Tightness of the maximum likelihood semidefinite relaxation for angular synchronization. 
\newblock {\em Mathematical Programming}, 163(1):145--167, 2017. 

\bibitem[BCSZ14]{BCSZ14}
Afonso S. Bandeira, Moses Charikar, Amit Singer, and Andy Zhu. 
\newblock Multi-reference alignment using semidefinite programming. 
\newblock In {\em Proceedings of the 5th Conference on Innovations in Theoretical Computer Science (ITCS)}, pages 459--470. Schloss Dagstuhl-Leibniz-Zentrumf\H{u}r Informatik, 2014.

\bibitem[BCLS20]{BCLS20}
Afonso S. Bandeira, Yutong Chen, Roy R. Lederman, and Amit Singer. 
\newblock Non-unique games over compact groups and orientation estimation in cryo-EM. 
\newblock {\em Inverse Problems}, 36(6):064002, 2020.

\bibitem[BEH+22]{BEH+22}
Afonso S. Bandeira, Ahmed El Alaoui, Samuel B. Hopkins, Tselil Schramm, Alexander S. Wein, and Ilias Zadik. 
\newblock The Franz-Parisi criterion and computational trade-offs in high dimensional statistics. 
\newblock In {\em Advances in Neural Information Processing Systems (NIPS)}, volume~35, pages 33831--33844. Curran Associates, Inc., 2022.

\bibitem[BKMR25+]{BKMR25+}
Afonso S. Bandeira, Anastasia Kireeva, Antoine Maillard, and Almut R\H{o}dder.
\newblock Randomstrasse101: Open problems of 2024.
\newblock arXiv preprint, arXiv:2504.20539.

\bibitem[BKW20]{BKW20}
Afonso S. Bandeira, Dmitriy Kunisky, and Alexander S. Wein.
\newblock Computational hardness of certifying bounds on constrained PCA problems.
\newblock In {\em Proceedings of the 11th Innovations in Theoretical Computer Science Conference (ITCS)}, pages 78:1--78:29. Schloss Dagstuhl-Leibniz-Zentrumf\H{u}r Informatik, 2020.

\bibitem[BHK+19]{BHK+19}
Boaz Barak, Samuel B. Hopkins, Jonathan Kelner, Pravesh K. Kothari, Ankur Moitra, and Aaron Potechin. 
\newblock A nearly tight sum-of-squares lower bound for the planted clique problem. 
\newblock {\em SIAM Journal on Computing}, 48(2):687--735, 2019.

\bibitem[BN11]{BN11}
Florent Benaych-Georges and Raj Rao Nadakuditi. 
\newblock The eigenvalues and eigenvectors of finite, low rank perturbations of large random matrices. 
\newblock {\em Advances in Mathematics}, 227(1):494--521, 2011.

\bibitem[BH22]{BH22}
Guy Bresler and Brice Huang. 
\newblock The algorithmic phase transition of random $k$-SAT for low degree polynomials. 
\newblock In {\em Proceedings of the IEEE 62nd Annual Symposium on Foundations of Computer Science (FOCS)}, pages 298--309. IEEE, 2022.

\bibitem[BHJK25]{BHJK25}
Rares-Darius Buhai, Jun-Ting Hsieh, Aayush Jain, and Pravesh K. Kothari.
\newblock The quasi-polynomial low-degree conjecture is false.
\newblock In {\em Proceedings of the IEEE 66th Annual Symposium on Foundations of Computer Science (FOCS)}. IEEE, 2025.

\bibitem[CDGL24+]{CDGL24+}
Guanyi Chen, Jian Ding, Shuyang Gong, and Zhangsong Li. 
\newblock A computational transition for detecting correlated stochastic block models by low-degree polynomials. 
\newblock {\em Annals of Statistics}, to appear.

\bibitem[CLS12]{CLS12}
Mihai Cucuringu, Yaron Lipman, and Amit Singer. 
\newblock Sensor network localization by eigenvector synchronization over the Euclidean group. 
\newblock {\em ACM Transactions on Sensor Networks}, 8(3):1--42, 2012.

\bibitem[dGJL07]{dGJL07}
Alexandre d'Aspremont, Laurent El Ghaoui, Michael I. Jordan, and Gert R.G. Lanckriet. 
\newblock A direct formulation for sparse PCA using semidefinite programming. 
\newblock {\em SIAM Review}, 49(3):434--448, 2007.

\bibitem[DAM16]{DAM16}
Yash Deshpande, Emmanuel Abbe, and Andrea Montanari. 
\newblock Asymptotic mutual information for the balanced binary stochastic block model. 
\newblock {\em Information and Inference: A Journal of the IMA}, 6(2):125--170, 2016.

\bibitem[DMW25]{DMW25}
Abhishek Dhawan, Cheng Mao, and Alexander S. Wein. 
\newblock Detection of dense subhypergraphs by low-degree polynomials. 
\newblock {\em Random Structures and Algorithms}, 66:e21279, 2025.

\bibitem[DD23]{DD23}
Jian Ding and Hang Du.
\newblock Matching recovery threshold for correlated random graphs.
\newblock {\em Annals of Statistics}, 51(4):1718--1743, 2023.

\bibitem[DDL25]{DDL25}
Jian Ding, Hang Du, and Zhangsong Li. 
\newblock Low-degree hardness of detection for correlated \ER graphs.
\newblock {\em Annals of Statistics}, 53(5):1833--1856, 2025.

\bibitem[DHSS25]{DHSS25}
Jingqiu Ding, Yiding Hua, Lucas Slot, and David Steurer.
\newblock Low degree conjecture implies sharp computational thresholds in stochastic block model.
\newblock In {\em Advances in Neural Information Processing Systems (NIPS)}, 2025.

\bibitem[DKW+22]{DKW+22}
Yunzi Ding, Dmitriy Kunisky, Alexander S. Wein, and Afonso S. Bandeira. 
\newblock Subexponential-time algorithms for sparse PCA. 
\newblock {\em Foundations of Computational Mathematics}, 22(1):1--50, 2022.

\bibitem[FP07]{FP07}
Delphine F\'eral and Sandrine P\'ech\'e. 
\newblock The largest eigenvalue of rank one deformation of large Wigner matrices. 
\newblock {\em Communications in Mathematical Physics}, 272(1):185--228, 2007.

\bibitem[GJW24]{GJW24}
David Gamarnik, Aukosh Jagannath, and Alexander S. Wein. 
\newblock Hardness of random optimization problems for Boolean circuits, low-degree polynomials, and Langevin dynamics. 
\newblock {\em SIAM Journal on Computing}, 53(1):1--46, 2024.

\bibitem[GZ19]{GZ19}
Tingran Gao and Zhizhen Zhao. 
\newblock Multi-frequency phase synchronization. 
\newblock In {\em Proceedings of the 36th International Conference on Machine Learning (ICML)}, pages 2132--2141. PMLR, 2019.

\bibitem[GK06]{GK06}
Arvind Giridhar and Praveen R. Kumar. 
\newblock Distributed clock synchronization over wireless networks: Algorithms and analysis. 
\newblock In {\em Proceedings of the 45th IEEE Conference on Decision and Control}, pages 4915--4920. IEEE, 2006.

\bibitem[GHL26+]{GHL26+}
Shuyang Gong, Dong Huang, and Zhangsong Li.
\newblock Fundamental limits of community detection in contextual multi-layer stochastic block models.
\newblock Forthcoming, 2026.

\bibitem[HTFF09]{HTFF09}
Trevor Hastie, Robert Tibshirani, Jerome H. Friedman, and Jerome H. Friedman. 
\newblock \emph{The elements of statistical learning: Data mining, inference, and prediction}. 
\newblock Springer Series in Statistics. Springer, 2009.

\bibitem[Hop18]{Hopkins18}
Samuel B. Hopkins. 
\newblock \emph{Statistical inference and the sum of squares method}. 
\newblock PhD thesis, Cornell University, 2018.

\bibitem[HKP+17]{HKP+17}
Samuel B. Hopkins, Pravesh K. Kothari, Aaron Potechin,  Prasad Raghavendra, Tselil Schramm, and David Steurer. 
\newblock The power of sum-of-squares for detecting hidden structures. 
\newblock In {\em Proceedings of the IEEE 58th Annual Symposium on Foundations of Computer Science (FOCS)}, pages 720--731. IEEE, 2017.

\bibitem[HS17]{HS17}
Samuel B. Hopkins and David Steurer. 
\newblock Efficient Bayesian estimation from few samples: Community detection and related problems. 
\newblock In {\em Proceedings of the IEEE 58th Annual Symposium on Foundations of Computer Science (FOCS)}, pages 379--390. IEEE, 2017.

\bibitem[JL09]{JL09}
Iain M. Johnstone and Yu A. Lu.  
\newblock On consistency and sparsity for principal components analysis in high dimensions. 
\newblock {\em Journal of the American Statistical Association}, 104(486):682--693, 2009.

\bibitem[KBK24+]{KBK24+}
Anastasia Kireeva, Afonso S. Bandeira, and Dmitriy Kunisky.
\newblock Computational lower bounds for multi-frequency group synchronization.
\newblock arXiv preprint, arXiv:2406.03424.

\bibitem[KVWX23]{KVWX23}
Pravesh K. Kothari, Santosh S. Vempala, Alexander S. Wein, and Jeff Xu.
\newblock Is planted coloring easier than planted clique?
\newblock In {\em Proceedings of the 36th Conference on Learning Theory (COLT)}, pages 5343--5372. PMLR, 2023.

\bibitem[Kun21]{Kunisky21}
Dmitriy Kunisky.
\newblock Hypothesis testing with low-degree polynomials in the Morris class of exponential families. 
\newblock In {\em Proceedings of the 34th Conference on Learning Theory (COLT)}, pages 2822--2848. PMLR, 2021.

\bibitem[KMW24]{KMW24}
Dmitriy Kunisky, Cristopher Moore, and Alexander S. Wein. 
\newblock Tensor cumulants for statistical inference on invariant distributions. 
\newblock In {\em Proceedings of the IEEE 65th Annual Symposium on Foundations of Computer Science (FOCS)}, pages 1007--1026. IEEE, 2024.

\bibitem[KWB22]{KWB22}
Dmitriy Kunisky, Alexander S. Wein, and Afonso S. Bandeira. 
\newblock Notes on computational hardness of hypothesis testing: Predictions using the low-degree likelihood ratio. 
\newblock In {\em Mathematical Analysis, its Applications and Computation: ISAAC}, pages 1--50. Springer, 2022.

\bibitem[KY24]{KY24}
Dmitriy Kunisky and Xifan Yu. 
\newblock Computational hardness of detecting graph lifts and certifying lift-monotone properties of random regular graphs. 
\newblock In {\em Proceedings of the IEEE 65th Annual Symposium on Foundations of Computer Science (FOCS)}, pages 1621--1633. IEEE, 2024.

\bibitem[LV07]{LV07}
John A. Lee and Michel Verleysen. 
\newblock \emph{Nonlinear dimensionality reduction}. 
\newblock Information Sciences and Statistics. Springer, 2007.

\bibitem[Li25]{Li25}
Zhangsong Li.
\newblock Algorithmic contiguity from low-degree conjecture and applications in correlated random graphs.
\newblock In {\em Approximation, Randomization, and Combinatorial Optimization. Algorithms and Techniques (APPROX/RANDOM)}, volume~353, pages 30:1--30:18. Schloss Dagstuhl-Leibniz-Zentrumf\H{u}r Informatik, 2025.

\bibitem[MW25a]{MW25b}
Cheng Mao and Alexander S. Wein. 
\newblock Optimal spectral recovery of a planted vector in a subspace.
\newblock {\em Bernoulli}, 31(2):1114--1139, 2025.

\bibitem[Mas14]{Mas14}
Laurent Massouli\'e. 
\newblock Community detection thresholds and the weak Ramanujan property. 
\newblock In {\em Proceedings of the 46th Annual ACM Symposium on Theory of Computing (STOC)}, pages 694--703. ACM, 2014.

\bibitem[MR14]{MR14}
Andrea Montanari and Emile Richard. 
\newblock Non-negative principal component analysis: Message passing algorithms and sharp asymptotics. 
\newblock {\em IEEE Transactions on Information Theory}, 62(3):1458--1484, 2014. 

\bibitem[MW25b]{MW25}
Andrea Montanari and Alexander S. Wein. 
\newblock Equivalence of approximate message passing and low-degree polynomials in rank-one matrix estimation. 
\newblock {\em Probability Theory and Related Fields}, 191(1-2):181--233, 2025.

\bibitem[MNS15]{MNS15}
Elchanan Mossel, Joe Neeman, and Allan Sly. 
\newblock Reconstruction and estimation in the planted partition model. 
\newblock {\em Probability Theory and Related Fields}, 162(3):431--461, 2015.

\bibitem[MNS18]{MNS18}
Elchanan Mossel, Joe Neeman, and Allan Sly.
\newblock A proof of the block model threshold conjecture. 
\newblock {\em Combinatorica}, 38(3):665--708, 2018.

\bibitem[PWBM18a]{PWBM18a}
Amelia Perry, Alexander S. Wein, Afonso S. Bandeira, and Ankur Moitra. 
\newblock Optimality and sub-optimality of PCA I: Spiked random matrix models. 
\newblock {\em Annals of Statistics}, 46(5):2416--2451, 2018.

\bibitem[PWBM18b]{PWBM18b}
Amelia Perry, Alexander S. Wein, Afonso S. Bandeira, and Ankur Moitra.
\newblock Message-passing algorithms for synchronization problems over compact groups.
\newblock {\em Communications on Pure and Applied Mathematics}, 71(11):2275--2322, 2018.

\bibitem[PBPB15]{PBPB15}
Jeffrey Russel Peters, Domenica Borra, Brad E. Paden, and Francesco Bullo.
\newblock Sensor network localization on the group of three-dimensional displacements. 
\newblock {\em SIAM Journal on Control and Optimization}, 53(6):3534--3561, 2015. 

\bibitem[SW22]{SW22}
Tselil Schramm and Alexander S. Wein. 
\newblock Computational barriers to estimation from low-degree polynomials. 
\newblock {\em Annals of Statistics}, 50(3):1833--1858, 2022.

\bibitem[Sin11]{Sin11}
Amit Singer. 
\newblock Angular synchronization by eigenvectors and semidefinite programming. 
\newblock {\em Applied and Computational Harmonic Analysis}, 30(1):20--36, 2011. 

\bibitem[SS11]{SS11}
Amit Singer and Yoel Shkolnisky. 
\newblock Three-dimensional structure determination from common lines in cryo-EM by eigenvectors and semidefinite programming. 
\newblock {\em SIAM Journal on Imaging Sciences}, 4(2):543--572, 2011.

\bibitem[WBP16]{WBP16}
Tengyao Wang, Quentin Berthet, and Yaniv Plan. 
\newblock Average-case hardness of RIP certification. 
\newblock In {\em Advances in Neural Information Processing Systems (NIPS)}, volume~29, pages 3826--3834. Curran Associates, Inc., 2016.

\bibitem[Wein22]{Wein22}
Alexander S. Wein. 
\newblock Optimal low-degree hardness of maximum independent set. 
\newblock {\em Mathematical Statistics and Learning}, 4(3-4):221--251, 2022.

\bibitem[Wein25+]{Wein25+}
Alexander S. Wein.
\newblock Computational complexity of statistics: New insights from low-degree polynomials.
\newblock arXiv preprint, arXiv:2506.10748.

\bibitem[WEM19]{WEM19}
Alexander S. Wein, Ahmed El Alaoui, and Cristopher Moore. 
\newblock The Kikuchi hierarchy and tensor PCA. 
\newblock In {\em Proceedings of IEEE 60th Annual Symposium on Foundations of Computer Science (FOCS)}, pages 1446--1468. IEEE, 2019.

\bibitem[YWF25]{YWF25}
Kaylee Y. Yang, Timothy L.H. Wee, and Zhou Fan.
\newblock Asymptotic mutual information in quadratic estimation problems over compact groups.
\newblock {\em Information and Inference: A Journal of the IMA}, 14(3): iaaf024, 2025. 

\bibitem[ZSWB22]{ZSWB22}
Ilias Zadik, Min Jae Song, Alexander S. Wein, and Joan Bruna. 
\newblock Lattice-based methods surpass sum-of-squares in clustering. 
\newblock In {\em Proceedings of the 35th Conference on Learning Theory (COLT)}, pages 1247--1248. PMLR, 2022.

\bibitem[ZHT06]{ZHT06}
Hui Zou, Trevor Hastie, and Robert Tibshirani. 
\newblock Sparse principal component analysis. 
\newblock {\em Journal of Computational and Graphical Statistics}, 15(2):186--265, 2006. 

\end{thebibliography}

\end{document}